\documentclass[12pt,reqno]{amsart}
\usepackage{epsfig}
\usepackage{graphicx}
\usepackage{mathrsfs}
\usepackage{verbatim}
\usepackage{amssymb}
\usepackage{latexsym}
\usepackage{amsmath}
\usepackage{amsfonts}
\usepackage{amsbsy}
\usepackage{amscd}
\usepackage[usenames]{color} 
\usepackage[mathscr]{eucal}
\usepackage{appendix}
\newtheorem{prop}{Proposition}[section]
\newtheorem{lem}[prop]{Lemma}

\newtheorem{defi}{Definition}[section]
\newtheorem{cor}[prop]{Corollary}
\newtheorem{thm}[prop]{Theorem}

\newtheorem{exam}{Example}[section]

\newtheorem{thmx}{Theorem}

\begin{document}
\baselineskip=17pt

\title[The spectrality of Cantor-Moran measure and Fuglede's Conjecture]
{The spectrality of Cantor-Moran measure and Fuglede's Conjecture}

\author{Jinsong Liu$^{1,2}$}
\author{Zheng-Yi Lu$^{\mathbf{*}1}$}
\author{Ting Zhou$^{3}$}
\address{$^1$ HLM, Academy of Mathematics and Systems Science, Chinese Academy of Sciences, Beijing
100190, China}
\address{$^2$ School of Mathematical Sciences, University of Chinese Academy of Sciences, Beijing 100049,
China}
\address{$^3$
School of Mathematics and Statistics, Hubei Key Laboratory of Mathematical
Sciences, Central China Normal University, Wuhan 430079, China}
\email{liujsong@math.ac.cn}
\email{zyluhnsd@163.com, zyluzky@amss.ac.cn}
\email{tingzhouhn@163.com}

\date{\today}
\keywords{Spectral Measure; Orthogonal basis; Moran Measure; Equi-positivity; Spectral set.\\
$^{\mathbf{*}}$\,\,  Corresponding author}
\subjclass[2010]{Secondary 28A78, 28A80; Primary 42C05, 42A85.}

\begin{abstract}
Let $\{(p_n, \mathcal{D}_n, L_n)\}$ be a sequence of Hadamard triples on $\mathbb{R}$. Suppose that the associated Cantor-Moran measure
$$
\mu_{\{p_n,\mathcal{D}_n\}}=\delta_{p_1^{-1}\mathcal{D}_1}\ast\delta_{(p_2p_1)^{-1}\mathcal{D}_2}\ast\cdots,
$$
where $\sup_n\{|p_n^{-1}d|:d\in \mathcal{D}_n\}<\infty$ and $\sup\#\mathcal{D}_n<\infty$.
It has been observed that the spectrality of $\mu_{\{p_n,D_n\}}$ is determined by equi-positivity.
A significant problem is what kind of Moran measures can satisfy this property.
In this paper, we introduce the conception of \textit{Double Points Condition Set} (\textit{DPCS}) to characterize the equi-positivity equivalently.
As applications of our characterization, we show that all singularly continuous Cantor-Moran measures are spectral.
For the absolutely continuous case, we study Fuglede's Conjecture on Cantor-Moran set.
We show that the equi-positivity of $\mu_{\{p_n,D_n\}}$ implies the tiling of its support, and
the reverse direction holds under certain conditions.
\end{abstract}

\maketitle
\section{\bf Introduction\label{sect.1}}

Let $\mu$ be a Borel probability measure on $\mathbb{R}^s$.
We call $\mu$ a spectral measure if there exists a countable set $\Lambda\subset\mathbb{R}^s$ such that $E_\Lambda:=\{e^{2\pi i\langle\lambda,x\rangle}:\lambda\in\Lambda\}$ forms an orthogonal basis for the Hilbert space $L^2(\mu)$, the space of all square-integrable functions with respect to $\mu$.
In this case, the countable set $\Lambda$ is called a spectrum of $\mu$.
Furthermore, if the spectral measure is the restriction of the Lebesgue measure on a bounded Borel subset $K$, we call $K$ a spectral set.
\\
\indent
A fundamental question in harmonic analysis is: what kind of measures are spectral measures?
The research on this problem was initiated by Fuglede \cite{fuglede}, whose famous conjecture(Fuglede's Conjecture) is that $\Omega$ is a spectral set on $\mathbb{R}^s$ if and only if $\Omega$ is a translational tile on $\mathbb{R}^s$, i.e., there exists a countable set $\mathcal{T}$ such that
$$
\Omega\oplus\mathcal{T}=\mathbb{R}^s.
$$
In this case, $\mathcal{T}$  is called the tiling set of $\Omega$.
The conjecture was proved to be wrong by Tao and others \cite{M.N., Tao} in $\mathbb{R}^s$ $(s\geq3)$.
However, it is unsolved in dimensions 1 and 2.
So far, it is still enlightening in the research of spectral measures\cite{FFLS19, La01, LL23, LW96, LM22}.

Recently, He, Lau and Lai \cite{HLL13} showed that $\mu$ is discrete with finite support, singularly continuous or absolutely continuous with respect to Lebesgue measure when the measure $\mu$ is spectral.
In 1998, Jorgensen and Pedersen \cite{Jorgenson-Pederson_1998} found the first singular and non-atomic spectral measure $\mu_{\frac14}$ (1/4-Cantor measure), which led to the study of singular spectral measures becoming a research hot spot \cite{Dai12, Dai16, DFY21, DS15, HL08, LLZ23, LDL22}.
Furthermore, the discovery of singular spectral measures produces many results which are different from the classical Fourier transform.
Strichartz \cite{Str06, Str00} showed that the mock Fourier series for each continuous function on this Cantor set converges uniformly to that function, which is better than the classical case.
But an interesting phenomenon is that the convergence of the mock Fourier series in $L^2(\mu_{\frac14})$ may be different for distinct spectra of $\mu_{\frac14}$ \cite{DHS14}.
\\
\indent
Self-affine measure is a kind of fundamental singular continuous measure, which is the probability measure uniquely generated by
$$
\mu_{R,D}(\cdot)=\frac{1}{\#D}\sum_{d\in D}\mu_{R,D}(R\cdot-d),
$$
where expanded matrix $R\in M_s(\mathbb{Z})$ and a finite digit set $D\subset\mathbb{Z}^s$.
As we know, a lot of self-affine spectral measures \cite{DHL14, Lai_2019, DL17, Laba} require that $(R,D,L)$ to satisfy Hadamard triple for some $L\subset\mathbb{Z}^s$, which is useful in the structure of the spectrum.

\begin{defi}\label{de1.1}
Let $R\in M_s(\mathbb{Z})$ be an expanded integer matrix.
Let $D, L\subset\mathbb{Z}^s$ be finite sets with $\#D=\#L\;(\#$ denotes the cardinality).
We say that system $(R,D,L)$ forms a Hadamard triple if the matrix
$$
H=\frac{1}{\sqrt{\#D}}[e^{2\pi i \langle R^{-1}d,\ell\rangle}]_{\ell\in L, d\in D}
$$ is unitary, $i.e.$, $H^\ast H=I$.
\end{defi}
In general, Strichartz \cite{Str00} first studied the spectrality of the Moran measure (generalized self-affine measure).
For a sequence of Hadamard triples $\{(R_n,B_n,L_n)\}_{n=1}^\infty$, he studied the spectrality of the associated Moran measure (under some conditions)
$$
\mu_{\{R_n,B_n\}}=\delta_{R_1^{-1}B_1}\ast\delta_{R_1^{-1}R_2^{-1}B_2}\ast\cdots
$$ with a compact support $$T=\left\{\sum_{j=1}^\infty(R_jR_{j-1}\cdots R_1)^{-1}b_j:b_j\in B_j\;\text{for\;all}\;j\right\},$$
where $\delta_A=\frac{1}{\#A}\sum_{a\in A}\delta_a$, and $\delta_a$ is the Dirac measure at $a$ and convergence is in weak sense.
The Moran measure $\mu_{\{R_n,B_n\}}$ is also called Cantor-Moran measure on $\mathbb{R}$.
In recent years, many positive results appeared in many papers \cite{AHH19, AH14, AHL15, CLSW20, LDZ22}, they proved that the Cantor-Moran measure $\mu_{\{R_n,B_n\}}$ is spectral when $\{(R_n,B_n,L_n)\}$ is a sequence of Hadamard triples with a strong assumption on $L_n$, or $\#B_n\leq4$.
Until 2019, An, Fu and Lai \cite{AFL19} introduced the concept of equi-positivity (we will explain this notion in Section \ref{sect.2}) to characterize the spectrality of Cantor-Moran measure $\mu_{\{R_n,B_n\}}$ when $B_n\subset \{0,1,\cdots,R_n-1\}$.

The following theorem is one of their main result.
\begin{thmx}[\cite{AFL19}]\label{thmxa}
Let ${(R_n, B_n, L_n)}$ be a sequence of Hadamard triples with
$$B_n\subset\{0,1,\cdots,R_n-1\}.$$
Suppose that $\underline{\lim}_{n\rightarrow\infty}\#B_n<\infty$.
Then there exists a equi-positive subsequence $\{\nu_{n_k}\}$ of $\{\nu_n\}$,
where
$$
\nu_n=\delta_{R_{n+1}^{-1}B_{n+1}}\ast\delta_{R_{n+1}^{-1}R_{n+2}^{-1}B_{n+2}}\ast\cdots.
$$
In particular, the associated Cantor-Moran measure $\mu_{\{R_n, B_n\}}$ is a spectral measure.
\end{thmx}
This result solves the spectrality of most Cantor-Moran measures supported inside the unit interval.
However, many cases are still unknown, such as the support of $\mu_{\{R_n, B_n\}}$ is outside $[0,1]$.
In the same paper, they showed that the spectral property of Moran measure is determined by the existence of equi-positive subsequence of $\{\nu_n\}$ under the assumption
\begin{equation}\label{(1.1')}
\sup_n\{|R_n^{-1}d|:d\in B_n\}<\infty.
\end{equation}
Simultaneously, they found that there exists a non-spectral measure $\mu_{\{2,B_n\}}$ satisfying $B_1=\{0,1\}$ and $B_n=3B_1(n\geq2)$, which also means that the equi-positive subsequence of $\{\nu_n\}$ may not exist under \eqref{(1.1')}.
We aim to address the following question.
\bigskip
\\
{\bf Question:} What Moran measures can guarantee the existence of equi-positive subsequence of $\{\nu_n\}$?
\bigskip

In this paper, we focus on the spectrality of the Cantor-Moran measures supported outside the unit interval under a natural assumption $\sup_n\#B_n<\infty$.
Our main goal is to study the existence of equi-positive subsequences of $\{\nu_n\}$.
To complete it, we introduce the concepts of \textit{Double Points Condition(DPC)} and \textit{DPC Set} to characterize the existence equivalently on $\mathbb{R}$(Theorem \ref{thm2.1}), which settles part of Conjecture 9.3 of \cite{AFL19}.
As an interesting application, we show that all singular continuous Cantor-Moran measures are spectral(Theorem \ref{thm2.2}).
For those absolutely continuous cases, we study the relation between equi-positivity and Fuglede's Conjecture on Cantor-Moran sets.
We find that a necessary condition for equi-positivity of $\mu_{\{R_n,B_n\}}$ is that the support of $\mu_{\{R_n,B_n\}}$ is a translational tile.
Under some assumptions, the necessary condition is also sufficient.

The main difficulties of this paper are reflected in the following aspects.
\par
$\bullet$ When we consider the equi-positivity of $\{\nu_n\}$ on $\mathbb{R}$, Proposition \ref{thm3.6} tells us that it is a crucial step to deal with the possible weak limit of $\{\nu_n\}$ when the integral periodic zero set of the weak limit is nonempty.
By \cite{AFL19}, we know the weak limit of $\{\nu_n\}$ can only be $\frac{1}{2}(\delta_{\{0\}}+\delta_{\{1\}})$ when $\text{spt}\nu_n\subset[0,1]$, but the weak limit may be varied in our case.
How to find the subsequence $\{\nu_{n_k}\}$ and determine its weak limit is a tough question.
To solve this problem, we use the diagonal principle to determine the form of the weak limit.
 Then by Argument Principle, Arzel\`a-Ascoli Theorem, Wiener theorem and some techniques, we obtain sufficient conditions for the existence of equi-positive subsequences of $\{\nu_n\}$.
\par
$\bullet$
It is complex to study the tiling property of Cantor-Moran sets.
We first show some Cantor-Moran sets are fundamental domains by \cite{LW97}, and then decompose the more general set into the union of several fundamental domains.
Finally, we use the properties of fundamental domains to characterize the structure of the tiling set.
\par
$\bullet$
In the process of studying the non-spectrality of Cantor-Moran measures, we flexibly apply the concept of selection mapping to characterize the structure of orthogonal sets.

\bigskip
We organize the paper as follows.

In the second section, we introduce some necessary preliminaries about \textit{DPC} and \textit{DPCS}, and state our main results.
Section \ref{sect.3} is devoted to the study of \textit{DPC}, \textit{DPCS} and equi-positivity.
In Section \ref{sect.4}, we characterize the equivalent conditions for the existence of equi-positive subsequences of $\{\nu_n\}$ on $\mathbb{R}$.
The spectral properties of singularly and absolutely continuous Cantor-Moran measures are left to Section \ref{sect.5}.
In Section \ref{sect.6}, we study the non-spectrality of Cantor-Moran measures.

\bigskip

\section{\bf Preliminaries and main results\label{sect.2}}
\setcounter{equation}{0}
In this section, we will give some basic notations for our paper.
In particular, we state our main results.
\subsection{\textit{ Preliminaries}\label{sect2.1}}
\

At the beginning of this section, recall that the Fourier transform of a Borel probability
measure $\mu$ on $\mathbb{R}$ is defined to be
$$
\widehat{\mu}(\xi)=\int_{\mathbb{R}} e^{2\pi i\langle x,\; \xi\rangle }d\mu(x), \quad \xi\in\mathbb{R}.
$$
Denote by ~$\mathscr{Z}(f):=\{x\in\mathbb{R}:f(x)=0\}$ the set of zeros of $f(x)$.
Note that $\{e^{2\pi i\langle\lambda,x\rangle}: \lambda \in \Lambda\}$ is an orthogonal set in $L^2(\mu)$ if and
only if $\widehat{\mu}(\lambda-\lambda')=0$ for any $\lambda\neq\lambda'\in\Lambda$, i.e.,
$$
(\Lambda-\Lambda)\setminus\{0\}\subset\mathscr{Z}(\widehat{\mu}(\xi)).
$$
In this case, $\Lambda$ is also called a bi-zero set of $\mu$.

Let
$\mathcal{P}(T)$ be the set of all Borel probability measures supported on compact set $T\subset\mathbb{R}.$
A sequence of measures $\mu_n \in \mathcal{P}(T), \: n=1,2,\cdots,$ is said to
converge weakly to a measure $\mu\in\mathcal{P}(T)$ ($\mu_{n}\xrightarrow{w}\mu$) if, for any $f\in C_0(\mathbb{R})$ (i.e. $f$ is continuous with compact support), we have
$$
\int fd\mu_n\rightarrow\int fd\mu,\;\;\text{as}\;n\rightarrow\infty.
$$
Let integer $|p_n|>1$, and let $\mathcal{D}_n\subset\mathbb{Z}$ be a simple digit set on $\mathbb{R}$ for all $n\geq1$.
We call $\mu_{\{p_n,\mathcal{D}_n\}}$ a Moran measure if
$$\mu_n=\delta_{p_1^{-1}\mathcal{D}_1}\ast\delta_{(p_2p_1)^{-1}\mathcal{D}_2}\ast\cdots\ast\delta_{(p_n\cdots p_1)^{-1}\mathcal{D}_n}\xrightarrow{w}\mu_{\{p_n,\mathcal{D}_n\}},$$ where
$\delta_D(x)=\begin{cases}\frac{1}{\#D}, & x\in D;
\\
0, & \text{otherwise}.
\end{cases}$
In this case, write
\begin{equation}\label{(2.1)}
\mu_{\{p_n,\mathcal{D}_n\}}=\delta_{p_1^{-1}\mathcal{D}_1}\ast\delta_{(p_2p_1)^{-1}\mathcal{D}_2}\ast\cdots.
\end{equation}
It is easy to check that
\begin{equation*}\label{(2.2)}
\widehat{\mu_{\{p_n,\mathcal{D}_n\}}}(\xi)=\prod_{j=1}^\infty \widehat{\delta_{\mathcal{D}_j}}((p_1p_{2}\cdots p_j)^{-1}\xi), \quad \xi\in\mathbb{R}.
\end{equation*}

In this paper, we always assume that there exists a compact set $T\subset\mathbb{R}$ such that the support $\text{spt}(\nu_n)\subset T$, where
\begin{equation}\label{(2.3)}
\nu_n=\delta_{p_{n+1}^{-1}\mathcal{D}_{n+1}}\ast\delta_{(p_{n+2}p_{n+1})^{-1}\mathcal{D}_{n+2}}\ast\cdots.
\end{equation}
This ensures the existence of $\nu_n$, and it also plays a pivotal role in the completeness characterization of the orthogonal set of $\mu_{\{p_n,\mathcal{D}_n\}}$.
An, Fu and Lai \cite{AFL19} give the concept of equi-positive family, which is important in characterization of the spectrality of $\mu_{\{p_n,\mathcal{D}_n\}}$.
\begin{defi}\label{de2.2}
Let $\Phi$ be a collection of probability measures on compact set $T\subset\mathbb{R}$.
We say that $\Phi$ is an equi-positive family(sequence) if there exists $\epsilon_0>0$ such that, for all $x\in[0,1)$, and for all $\nu\in\Phi$, there exists
$k_{x,\nu}\in\mathbb{Z}$ with $$|\widehat{\nu}(x+k_{x,\nu})|>\epsilon_0.$$
\end{defi}
For convenience, we call that $\mu_{\{p_n,\mathcal{D}_n\}}$ satisfies {\textit{the equi-positive condition}}, if there exist a compact set $T\subset\mathbb{R}$ and an equi-positive sequence $\{\nu_{n_k}\}\subset\{\nu_n\}$ such that all support of $\nu_{n_k}$ are in $T$.
In \cite{AFL19}, An et al. gave the following theorem, which is a basic criterion.
\begin{thmx}[\cite{AFL19},\cite{LDZ22}]\label{thmxb}
Let $\{(p_n, \mathcal{D}_n, L_n)\}$ be a sequence of Hadamard triples on $\mathbb{R}$.
Suppose that $\mu_{\{p_n,\mathcal{D}_n\}}$ satisfies \textit{the equi-positive condition}.
Then the associated Cantor-Moran measure $\mu_{\{p_n, \mathcal{D}_n\}}$ is a spectral measure and it always admits a spectrum $\Lambda\subset\mathbb{Z}$.
\end{thmx}

\subsection{\textit{ Spectrality of Cantor-Moran measures}\label{sect2.2}}
\

According to Theorem \ref{thmxb}, the spectral properties of most Cantor-Moran measures in $\mathbb{R}$ depend on equi-positivity of $\{\nu_n\}$.
In this subsection, we give an equivalent characterization of equi-positivity.
Now we introduce some new concepts.
\begin{defi}[\textit{DPCS, DPC}]\label{de2.3}
Suppose that $D\subset \mathbb{Z}$ and $p>1$.
Define $T_{p,D}\subset[0,1)$ to be the \textit{Double Points Condition Set (DPCS)} of $(p,D)$, i.e., $T_{p,D}$ is the set of all points $\xi$ with the given property: there exist $\ell_1,\ell_2\in \{0,1,\cdots,p-1\}$ such that
$\widehat{\delta_{p^{-1}D}}(\xi+\ell_i)\neq0$ for $i=1,2$.
Moreover, we call that $(p,D)$ satisfies \textit{Double Points Condition (DPC)} if $T_{p,D}=[0,1)$.
\end{defi}
On $\mathbb{R}$, it is easy to check that there exists a compact set $T$ such that all support of $\nu_n$ are in $T$ if and only if $\sup_n\sup_{d_n\in D_n}|p_n^{-1}d_n|<\infty$. For convenience, we always assume that $p_n>0$ and $0\in\mathcal{D}_n\subset\mathbb{N}$ for all $n\geq1$.
\begin{thm}\label{thm2.1}
Let $\{(p_n, \mathcal{D}_n, L_n)\}$ be a sequence of Hadamard triples on $\mathbb{R}$, and let $\mu_{\{p_n,\mathcal{D}_n\}}$, $\nu_n$ be given as \eqref{(2.1)} and \eqref{(2.3)}, respectively.
Suppose that
$$\sup_n\{p_n^{-1}d:d\in \mathcal{D}_n\}<\infty\;\;\text{and}\;\;\sup_n\#\mathcal{D}_n<\infty.$$
Then the following statements are equivalent:
\\
(i)\;\;\;there exists an equi-positive sequence $\{\nu_{n_k}\}$;
\\
(ii)\;\;$\limsup_{k\rightarrow\infty}T_{p_k,\mathcal{D}_k}\supset(0,1)$, where $T_{p_k,\mathcal{D}_k}$ is a DPCS of $(p_k,\mathcal{D}_k)$;
\\
(iii)\;at least one of the following two conditions is satisfied: $\#\{n:p_n>\#\mathcal{D}_n\}=\infty$ or $\gcd_{j\geq n}\{\gcd \mathcal{D}_j\}\equiv1,\;\;\forall n\geq1$.
\\
In particular, the associated Cantor-Moran measure $\mu_{\{p_n, \mathcal{D}_n\}}$ is a spectral measure and it always admits a spectrum $\Lambda\subset\mathbb{Z}$.
\end{thm}
Meanwhile, the theorem partially settles Conjecture 9.3 of \cite{AFL19}, which conjectures that the equipositivity and spectrality are equivalent under assumption $\gcd\mathcal{D}_n\equiv1$.
The following theorem is the connection between the singularity and spectrality of Cantor-Moran measure.
\begin{thm}\label{thm2.2}
Suppose that $\{(p_n, \mathcal{D}_n, L_n)\}$ is a sequence of Hadamard triples with $\sup_n\{p_n^{-1}d:d\in \mathcal{D}_n\}<\infty$ and $\sup_n\#\mathcal{D}_n<\infty$, and let Cantor-Moran measure
$$
\mu_{\{p_n,\mathcal{D}_n\}}=\delta_{p_1^{-1}\mathcal{D}_1}\ast\delta_{p_1^{-1}p_2^{-1}\mathcal{D}_2}\ast\cdots.
$$
Then $\mu_{\{p_n,\mathcal{D}_n\}}$ is singularly continuous with respect to Lebesgue measure if and only if there are infinitely many pairs $(p_n,\mathcal{D}_n)$ satisfying DPC, i.e. $\#\{n:p_n>\#\mathcal{D}_n\}=\infty$(See Lemma \ref{thm3.2}).
In particular, all singularly continuous Cantor-Moran measures are spectral.
\end{thm}

In the following, we consider the equi-positivity of $\mu_{\{p_n,\mathcal{D}_n\}}$ when the measure is an absolutely continuous with respect to Lebesgue measure.
By \cite{Dut_Lai}, the absolutely continuous case is closely related to Fuglede's Conjecture.
A significant research is to characterize the equi-positivity and the tiling properties of Cantor-Moran set.

Before describing the result, we need to introduce the concept of the no-overlap condition.
Given $\{(p_n,\mathcal{D}_n)\}$ that generates a Cantor-Moran measure $\mu_{\{p_n,\mathcal{D}_n\}}$, we write $\mu_{\{p_n,\mathcal{D}_n\}}=\mu_n\ast\mu_{>n}$.
Denote by $K_n$, $K_{>n}$ to be the support of $\mu_n$ and $\mu_{>n}$ respectively.
We say that $\mu_{\{p_n,\mathcal{D}_n\}}$ satisfies the no-overlap condition if
$$
\mu_{\{p_n,\mathcal{D}_n\}}((b+K_{>n})\cap(b'+K_{>n}))=0
$$
for all $b\neq b'\in K_n$.
\begin{thm}\label{thm2.3}
Let $\{(p_n, \mathcal{D}_n, L_n)\}$ be a sequence of Hadamard triples with $\sup_n\{p_n^{-1}d:d\in\mathcal{D}_n\}<\infty$ and $\sup_n\#\mathcal{D}_n<\infty$, and let Cantor-Moran measure
$$
\mu_{\{p_n,\mathcal{D}_n\}}=\delta_{p_1^{-1}\mathcal{D}_1}\ast\delta_{p_1^{-1}p_2^{-1}\mathcal{D}_2}\ast\cdots,
$$
where $\gcd_{j\geq 1}\{\gcd \mathcal{D}_j\}=1$.
Suppose that there exists integer $d\geq1$ such that $\gcd_{j\geq n}\{\gcd\mathcal{D}_j\}\equiv d$ and $p_n=\#\mathcal{D}_n$ for any $n\geq2$.
Then $\mu_{\{p_n,\mathcal{D}_n\}}$ is an absolutely continuous with respect to the Lebesgue measure.
If there exists an equi-positive subsequence $\{\nu_{n_k}\}$, then $\mu_{\{p_n,\mathcal{D}_n\}}$ satisfies the no-overlap condition and its support is a translational tile.
Conversely, if there exists $n$ such that the tiling of $\text{spt}(\nu_n)$ is unique(difference is a constant), then the necessary conditions are also sufficient.
\end{thm}

\textbf{Remark}:
By combining the first $n$ terms, we can always write $\mu_{\{p_n,\mathcal{D}_n\}}$ in the form that satisfies hypothesis $p_n=\#\mathcal{D}_n$ and $\gcd_{j\geq n}\{\gcd\mathcal{D}_j\}\equiv d$ for any $n\geq2$.
The assumption $\gcd_{j\geq 1}\{\gcd \mathcal{D}_j\}=1$ ensures the possible equi-positivity, and does not effect the spectral or tiling properties of Cantor-Moran set.

Note that the equi-positivity of Moran measure $\mu_{\{p_n,\mathcal{D}_n\}}$ may be invalid under hypothesis
\begin{equation}\label{(2.4)}
\sup_n\{p_n^{-1}d:d\in D_n\}=\infty.
\end{equation}
For the study of spectral property, this condition is still inseparable, which will be explained in the following proposition.
\begin{prop}\label{thm2.8}
The following statements hold.
\\
(i)\;\;\;There exists a Cantor-Moran non-spectral measure satisfying \eqref{(2.4)}.
\\
(ii)\;\;There exists a Cantor-Moran spectral measure satisfying \eqref{(2.4)}.
\end{prop}

\bigskip

\section{\bf DPC and equi-positivity\label{sect.3}}
\setcounter{equation}{0}
In this section, we will set up some brief and useful results for the rest of this paper.
In the first subsection, we characterize the equivalence conditions of \textit{DPC} in $\mathbb{R}$.
In particular, we prove (ii) $\Leftrightarrow$ (iii) of Theorem \ref{thm2.1}.
The second subsection develops the theory of equi-positivity of general measures and establishes the relation between that and weak convergence.
Finally, we give a useful estimate at the end of this section.

\subsection{\textit{ Equivalent conditions for DPC and DPCS in }$\mathbb{R}$\label{sect3.1}}
\

As for prerequisites, the reader is expected to be familiar with the properties of Hadamard triple.
\begin{lem}[\cite{Lai_2019}]\label{thm3.1}
Let $(R,D,L)$ be a Hadamard triple on $\mathbb{R}^s$.
Then the following statements hold.
\\
(i)\;\;For any $\xi\in\mathbb{R}^s$, one has
$$\sum_{\ell\in L}|\widehat{\delta_{R^{-1}D}}(\xi+\ell)|^2=1;
$$
\\
(ii)\;\;$(R,D,\widetilde{L})$ is a Hadamard triple, where $\widetilde{L}\equiv L\pmod {R^*\mathbb{Z}^s}$;
\\
(iii)\;$L$ is a spectrum of $\delta_{R^{-1}D}$.
\end{lem}
Let $(p,D,L)$ be a Hadamard triple on $\mathbb{R}$.
For convenience, we always assume that $p>0$, $0\in D\subset\mathbb{N}$ and $0\in L\subset\{0,1,\cdots,p-1\}:=L^*$.
\begin{lem}\label{thm3.2}
Let $(p,\mathcal{D},L)$ be a Hadamard triple on $\mathbb{R}$.
Then $(p,\mathcal{D})$ satisfies DPC if and only if $p>\#\mathcal{D}$.
\end{lem}
\begin{proof}
We first prove the sufficiency.
Since $(p,\mathcal{D},L)$ is a Hadamard triple and (i) of Lemma \ref{thm3.1}, we have
\begin{equation}\label{(3.1)}
\sum_{\ell\in L}\left|\widehat{\delta_{\mathcal{D}}}\left(\frac{x+\ell}{p}\right)\right|^2\equiv1,
\end{equation}
which implies that $\widehat{\delta_{\mathcal{D}}}(\frac{x+\ell_1}{p})\neq0$ for some $\ell_1\in L$.

If $|\widehat{\delta_{\mathcal{D}}}(\frac{x+\ell_1}{p})|<1$, by \eqref{(3.1)}, it follows that there exists $\ell_2\in L\setminus\{\ell_1\}$ with $\widehat{\delta_{\mathcal{D}}}(\frac{x+\ell_2}{p})\neq0$,
which implies that the sufficiency follows.

If $|\widehat{\delta_{\mathcal{D}}}(\frac{x+\ell_1}{p})|=1$, i.e., $|\sum_{d\in\mathcal{D}}e^{2\pi i\frac{(x+\ell_1)d}{p}}|=\#\mathcal{D}$,
then $\frac{(x+\ell_1)d}{p}=\frac{(x+\ell_1)d'}{p}\pmod 1$ for any $d\neq d'$.
According to $0\in\mathcal{D}$, one has
$\frac{(x+\ell_1)d}{p}\in\mathbb{Z}$ for all $d\in \mathcal{D}$.
Hence
\begin{equation}\label{(3.2)}
\frac{x+\ell_1}{p}\in\frac{1}{\gcd \mathcal{D}}\mathbb{Z}.
\end{equation}
Let $\gcd \mathcal{D}=\mathbf{d}\geq1$.
It follows from \eqref{(3.2)} that
$x=\frac{pz_1}{\mathbf{d}}-\ell_1$ for some $z_1\in\mathbb{Z}.$
Supposing that $\widehat{\delta_{\mathcal{D}}}(\frac{x+\ell}{p})=0$ for any $\ell\in L^*\setminus \{\ell_1\}$, it implies
$$\left|\widehat{\delta_{\mathcal{D}}}\left(\frac{x+\ell}{p}\right)\right|=\left|\widehat{\delta_{D}}\left(\frac{z_1}{\mathbf{d}}+\frac{\ell-\ell_1}{p}\right)\right|
=\left|\widehat{\delta_{\mathcal{D}}}\left(\frac{\ell-\ell_1}{p}\right)\right|=0.$$
Noting that $\widehat{\delta_{\mathcal{D}}}(x)$ is a function with an integer as the period, hence $\widehat{\delta_{\mathcal{D}}}(\frac{i}{p})=0$ for all $i\in\{1,\cdots,p-1\}$.
Therefore
$$
\frac{1}{p}(L^*-L^*)\setminus\{0\}\subset\mathscr{Z}(\widehat{\delta_{\mathcal{D}}}),
$$
$i.e.$, $L^*$ is a bi-zero set of $\delta_{p^{-1}\mathcal{D}}$.
From $L\subset L^*$, one obtains $L=L^*$, which contradicts with the condition $\#L=\#\mathcal{D}<p=\#L^*$.

In the following, we prove the necessity by a contradiction.
Suppose that $p=\#\mathcal{D}$.
Due to $(p,\mathcal{D},L)$ is a Hadamard triple and (ii) of Lemma \ref{thm3.1}, write $L=\{0,1,\cdots,p-1\}$.
Note that
$$
\sum_{\ell\in L}|\widehat{\delta_{p^{-1}\mathcal{D}}}(\xi+\ell)|^2=1.
$$
Choose $\xi_0=0$.
It is easy to see
$\widehat{\delta_{p^{-1}\mathcal{D}}}(\xi_0)=1$ and
$$
\widehat{\delta_{p^{-1}\mathcal{D}}}(\xi_0+\ell)=0
$$
for all $\ell\in L\setminus\{0\}$, which implies that $0\not\in T_{p,\mathcal{D}}$.
\end{proof}
As we all know, $p\geq \#\mathcal{D}$ when $(p,\mathcal{D},L)$ forms a Hadamard triple.
In this case, we study the property of $T_{p,\mathcal{D}}$ if $p=\#\mathcal{D}$.
\begin{lem}\label{thm3.3}
Let $(p,\mathcal{D},L)$ be a Hadamard triple with $p=\#\mathcal{D}$.
Then $T_{p,\mathcal{D}}$ is a DPCS of $(p,\mathcal{D})$ if and only if $T_{p,\mathcal{D}}=[0,1)\setminus \frac{1}{\gcd \mathcal{D}}\mathbb{Z}$.
\end{lem}
\begin{proof}
Without loss of generality, we assume $L=\{0,1,\cdots,p-1\}$.
It follows that, for any $\xi\in\mathbb{R}$,
\begin{equation}\label{(3.3)}
\sum_{\ell\in L}|\widehat{\delta_{p^{-1}\mathcal{D}}}(\xi+\ell)|^2\equiv1.
\end{equation}
We assert that $\gcd(\gcd \mathcal{D},p)=1$.
Suppose that $p=dp'$ and $\mathcal{D}=d\mathcal{D}'$, where $\gcd(\gcd \mathcal{D},p)=d>1$.
Note that $L$ is a bi-zero set of $\delta_{p^{-1}\mathcal{D}}=\delta_{p'^{-1}\mathcal{D}'}$.
Taking $\ell_1=0$, $\ell_2=p'\in L$, this implies that
$$0=\widehat{\delta_{p'^{-1}\mathcal{D}'}}(\ell_2-\ell_1)=\widehat{\delta_{\mathcal{D}'}}(1)=1.
$$
It is impossible, so the claim follows.
If $\xi=\frac{z}{\gcd \mathcal{D}}$ for some integer $z$,
then by the claim, there exists $\ell\in L$ such that $z+\ell\gcd \mathcal{D}=0\pmod p$.
Hence
$$
\widehat{\delta_{p^{-1}\mathcal{D}}}(\xi+\ell)=\widehat{\delta_{\mathcal{D}}}((p\gcd \mathcal{D})^{-1}(z+\ell\gcd \mathcal{D}))=1.
$$
From \eqref{(3.3)} and the definition of $T_{p,\mathcal{D}}$, one has $T_{p,\mathcal{D}}\subset[0,1)\setminus\frac{1}{\gcd \mathcal{D}}\mathbb{Z}$.
On the other hand, if $
|\widehat{\delta_{p^{-1}\mathcal{D}}}(\xi+\ell)|=1
$
for some $\xi\in[0,1)$ and $\ell\in L$,
then $p^{-1}d(\xi+\ell)\in\mathbb{Z}$ for all $d\in \mathcal{D}$.
This yields that
$\xi+\ell\in \frac{p}{\gcd\mathcal{D}}\mathbb{Z}$, i.e. $\xi\in\frac{1}{\gcd\mathcal{D}}\mathbb{Z}$,
which completes the proof.
\end{proof}
{\textbf{Proof of Theorem \ref{thm2.1}: (ii) $\Leftrightarrow$ (iii)}}

(ii) $\Rightarrow$ (iii): We prove the result by a contradiction.
If (iii) doesn't hold, we assume that there exist integers $N$ and $d>1$ such that $p_n=\#\mathcal{D}_n$ and $d|\gcd \mathcal{D}_n$ for all $n\geq N$.
Together with Lemma \ref{thm3.3}, it implies that $\frac1d\not\in T_{p_n,\mathcal{D}_n},\;\forall n\geq N$.
Hence $\frac1d\not\in \cup_{k\geq n}T_{p_k,\mathcal{D}_k}$ for any $n\geq N$, a contradiction.
This proves the result.

(iii) $\Rightarrow$ (ii): If $\#\{j:p_j>\#\mathcal{D}_j\}=\infty$, by using Lemma \ref{thm3.2} and definition of \textit{DPC}, we obtain that $\cup_{k\geq n}T_{p_k,\mathcal{D}_k}\supset[0,1)$ for all $n\geq N$ and some $N\geq1$.
Hence $\limsup_{k\rightarrow\infty} T_{p_k,\mathcal{D}_k}=\cap_{n=1}^\infty\cup_{k\geq n}T_{p_k,\mathcal{D}_k}\supset[0,1)$.
For another case, suppose that there exists $N$ such that $p_n=\#\mathcal{D}_n$ and $\gcd_{j\geq n}\{\gcd \mathcal{D}_j\}=1$ for all $n\geq N$.
It follows from Lemma \ref{thm3.3} that
$$\cup_{k\geq n}T_{p_k,\mathcal{D}_k}=\cup_{k\geq n}([0,1)\setminus\frac{1}{\gcd \mathcal{D}_k}\mathbb{Z})=[0,1)\setminus(\cap_{k\geq n}\frac{1}{\gcd \mathcal{D}_k}\mathbb{Z})=(0,1),
$$
which gives the desired result.

\subsection{\textit{Equi-positivity and weak convergence}\label{sect3.2}}
\

In this subsection, we obtain the properties of equi-positivity and weak convergence.
The following well-known results are useful in this paper.
\begin{thm}[\cite{jiqiao}]\label{thm3.4}
Let $\mu_1,\mu_2,\cdots$ be locally finite Borel measures on $\mathbb{R}^s$ with $\sup_n\mu_n(A)<\infty$ for all bounded set $A$.
Then $\{\mu_n\}_{n=1}^\infty$ has a weakly convergent subsequence.
\end{thm}

\begin{lem}[\cite{K.L.}]\label{thm3.5}
Let $\mu,\mu_1,\mu_2,\cdots$ be finite Borel measures on $\mathbb{R}^s$.
Then the following are equivalent.
\\
$(i)$\;\;\;$\{\mu_n\}$ converges weakly to a finite Borel measure $\mu$.
\\
$(ii)$\;\;For any open set $O$, $\mu(O)\leq \underline{\lim}_{n\rightarrow\infty}\mu_n(O)$.
\\
$(iii)$\;For any compact set $K$, $\mu(K)\geq \overline{\lim}_{n\rightarrow\infty}\mu_n(K)$.
\\
$(iv)$\;\;The Fourier transform $\{\widehat{\mu_n}(\xi)\}$ converges to $\widehat{\mu}(\xi)$ uniformly on all compact subsets
of $\mathbb{R}^s$.
\end{lem}

For any $x\in\mathbb{R}^s$, denote the ball $B(x;\delta):=\{y\in\Bbb R^s:\text{dist}(x,y)<\delta\}$.
For any set $A\subset\mathbb{R}^s$, write $B(A;\delta):=\{y\in\Bbb R^s:\text{dist}(A,y)<\delta\}$.
\begin{prop}\label{thm3.6}
Let $\mu$ be a Borel probability measure with a compact support on $\mathbb{R}^s$, and let the integral periodic zero set
$$\mathcal{Z}(\mu)=\{\xi\in[0,1)^s:\widehat{\mu}(\xi+k)=0,\;\text{for\;all}\;k\in\mathbb{Z}^s\}.$$
For any $\delta>0$, then
$$
\inf_{x\in [0,1]^s\setminus B(\mathcal{Z}(\mu);\delta)}\sup_{k\in\mathbb{Z}^s}|\widehat{\mu}(x+k)|>0.
$$
If $\mu_{n}\xrightarrow{w}\mu$, then for any $\delta>0$, there exists $N$ such that, for all $n\geq N$,
$$
\inf_{x\in [0,1]^s\setminus B(\mathcal{Z}(\mu);\delta)}\sup_{k\in\mathbb{Z}^s}|\widehat{\mu_n}(x+k)|>0.
$$

\end{prop}
\begin{proof}
We prove the lemma by finite covering theorem.
For any $x\in [0,1]^s\setminus B(\mathcal{Z}(\mu);\delta)$, by the definition of $\mathcal{Z}(\mu)$, we obtain
$$
|\widehat{\mu}(x+k_x)|\geq \epsilon_{x}
$$
for some $k_x\in\mathbb{Z}^s$ and $\epsilon_{x}>0$.
The continuity gives that
\begin{equation}\label{(3.4)}
|\widehat{\mu}(y+k_x)|\geq \frac12\epsilon_{x},\;\forall y\in B(x;\delta_x)
\end{equation}
for some $\delta_x>0$.
Note that $[0,1]^s\setminus B(\mathcal{Z}(\mu);\delta)$ is a compact set and
$$\cup_{x\in [0,1]^s\setminus B(\mathcal{Z}(\mu);\delta)}B(x;\delta_{x})\supset [0,1]^s\setminus B(\mathcal{Z}(\mu);\delta).$$
There exist balls $\{B(x_i;\delta_{x_i})\}_{i=1}^m$ satisfying
$$
\cup_{i=1}^mB(x_i;\delta_{x_i})\supset [0,1]^s\setminus B(\mathcal{Z}(\mu);\delta).
$$
Let $\epsilon:=\frac12\min_{i}\epsilon_{x_i}$.
For any $x\in[0,1]^s\setminus B(\mathcal{Z}(\mu);\delta)$, there exists $i$ such that $x\in B(x_i;\delta_{x_i})$.
It follows from \eqref{(3.4)} that
\begin{equation*}
|\widehat{\mu}(x+k_{x_i})|\geq \frac12\epsilon_{x_i}\geq \epsilon>0,
\end{equation*}
which gives the desired inequality.

Now assume that $\mu_{n}\xrightarrow{w}\mu$.
It follows from \eqref{(3.4)} and (iv) of Lemma \ref{thm3.5} that there exists $N_i>0$ such that
$$
|\widehat{\mu_n}(x+k_{x_i})|\geq \frac14\epsilon_{x_i},\;\:\: \forall x\in B(x_i;\delta_{x_i})
$$
for all large $n\geq N_i$.
Finally, the conclusion follows from the finiteness of the number of $\{x_i\}$ and $N=\max N_i$.
\end{proof}
By the above conclusion, the following result holds directly.
\begin{lem}\label{thm3.7}
Let $\{\mu_n\}$ and $\mu$ be Borel probability measures with compact supports on $\mathbb{R}^s$.
Suppose that $\mu_n\xrightarrow{w}\mu$.
If $\mathcal{Z}(\mu)=\emptyset$, then there exists an equi-positive subsequence $\{\mu_{n_k}\}$ of $\{\mu_n\}$.
\end{lem}
\begin{thm}\label{thm3.8}
Let $\{\mu_n\}$ and $\mu$ be Borel probability measures with compact supports on $\mathbb{R}^s$.
Suppose that $\mu_{n}\xrightarrow{w}\mu$.
Then $\mathcal{Z}(\mu)=\emptyset$ if and only if there exist $\varepsilon,N>0$ such that $\mathcal{Z}(\mu_n)=\emptyset$ for $n\geq N$ and
\begin{equation}\label{(3.5*)}
\sup_{n\geq N,\xi\in[0,1)^s}\min\{|k_\xi|:|\widehat{\mu_{n}}(\xi+k_\xi)|\geq\varepsilon\}<\infty.
\end{equation}
In particular, if $\mu_n$ can be written as $\mu_n=\delta_1\ast\cdots\ast\delta_n$, where $\delta_j$ is a discrete probability measure with $\sup_n\#\text{spt}\delta_n<\infty$ and $\text{spt}\mu\subset K$ for some compact set $K$, then $\mathcal{Z}(\mu)=\emptyset$ if and only if
there exist $N>0$ such that $\mathcal{Z}(\mu_n)=\emptyset$ for $n\geq N$ and
$$
\sup_{n\geq N,\xi\in[0,1)^s}\min\{|k_\xi|:\widehat{\mu_{n}}(\xi+k_\xi)\neq0\}<\infty.
$$
\end{thm}
\begin{proof}
We prove the necessity of the first conclusion.
Assume by contradiction that $\mathcal{Z}(\mu)\neq\emptyset$.
Let $\xi\in\mathcal{Z}(\mu)$.
The assumption \eqref{(3.5*)} implies that, for all $n\geq N$ there exists $k_{n,\xi}\in\mathbb{Z}^s$ such that $|k_{n,\xi}|<C$ and
$$
|\widehat{\mu_{n}}(\xi+k_{n,\xi})|\geq\varepsilon>0,
$$
where $C$ is the supremum of \eqref{(3.5*)}.
Choose a subsequence $\{n_j\}_{j=1}^\infty\subset\mathbb{N}$ such that
$$
|\widehat{\mu_{n_{j}}}(\xi+k_{n_{j},\xi})|\geq\varepsilon>0
$$
and $k_{n_{j},\xi}=k_{n_{j'},\xi}$ for any $j\neq j'$.
Hence, by (iv) of Lemma \ref{thm3.4},
$$
|\widehat{\mu}(\xi+k_{n_{1},\xi})|=\lim_{j\rightarrow\infty}|\widehat{\mu_{n_{j}}}(\xi+k_{n_{j},\xi})|\geq\varepsilon>0.
$$
This yields that $\xi\not\in\mathcal{Z}(\mu)$, a contradiction.

We next consider the sufficiency.
For any $\xi\in\mathbb{R}^s$, there exists $k_\xi\in\mathbb{Z}$ such that
$$
\widehat{\mu}(\xi+k_\xi)\neq0.
$$
Then there exists a ball $B(\xi;\delta_x)$ such that
$$
|\widehat{\mu}(y+k_\xi)|\geq\epsilon_\xi>0,\;\:\: \forall y\in B(\xi;\delta_x).
$$
By using the fact $[0,1]^s\subset\cup_{\xi\in[0,1]^s}B(\xi;\delta_{\xi})$ and finite covering theorem, there exist $\{\xi_i\}_{i=1}^m$ such that
$[0,1]^s\subset\cup_{j=1}^m B(\xi_j,\delta_{\xi_j})$.
Hence for any $\xi\in[0,1]^s$, there exists $B(\xi_j;\delta_{\xi_j})\ni\xi$ such that
$$
|\widehat{\mu}(\xi+k_{\xi_i})|\geq\epsilon_0>0,
$$
where $\epsilon_0:=\min_{1\leq j\leq m}\epsilon_{\xi_j}$.
Thus by (iv) of Lemma \ref{thm3.5}, there exists $N>0$ such that $|\widehat{\mu}_{n}(\xi+k_{\xi_i})|\geq\frac12\epsilon_0$ for all $n\geq N$.

In the following, we assume that $\mu_n=\delta_1\ast\cdots\ast\delta_n$, where $\delta_j$ is a discrete
probability measure with $\sup_n\#\text{spt}\delta_n<\infty$.
From the above conclusion, we would need to show \eqref{(3.5*)}.
The existence of $\mu$ allows us to assume $\mu_{>n}:=\delta_{n+1}\ast\delta_{n+2}\ast\cdots$ exists and $\mu=\mu_n\ast\mu_{>n}$.
By Corollary 1.2 in \cite{LMW22}, one has $\sum_{j=1}^\infty\max\{|a|:a\in\text{spt}\delta_j\}<\infty$.
Therefore $\lim_{n\rightarrow\infty}\sum_{j=n}^\infty\max\{|a|:a\in\text{spt}\delta_j\}=0$.
This also yields that $\lim_{n\rightarrow\infty}\max\{|a|:a\in\text{spt}\mu_{>n}\}=0$.
Then, for any compact set $K'\subset\mathbb{R}^s$ and $\xi\in K'$,
there exists $N_{K'}>0$ such that $\max\{|a|:a\in\text{spt}\mu_{>n}\}\max\{|a|:a\in K'\}\leq \frac16$ for any $n\geq N_{K'}$.
Hence
$$
|\widehat{\mu_{>n}}(\xi)|=\left|\int e^{2\pi i\xi x}d\mu_{>n}(x)\right|\geq
\left|\int \cos2\pi \xi xd\mu_{>n}(x)\right|\geq \frac12
$$
for any $n\geq N_{K'}$.
Note that the condition
$$
\sup_{n\geq N,\xi\in[0,1)^s}\min\{|k_\xi|:\widehat{\mu_{n}}(\xi+k_\xi)\neq0\}<\infty
$$
implies that
$\{\xi+k_\xi:\xi\in[0,1]^s\}:=K'$ is uniformly bounded.
We can find an integer $N'\geq\max\{N,N_{K'}\}$ such that for any $\xi\in[0,1]$, there exists $k_\xi$ satisfying
$|\widehat{\mu_{N'}}(\xi+k_\xi)|\geq \epsilon_\xi$ and $
|\widehat{\mu_{>N'}}(\xi+k_\xi)|\geq \frac12.
$
By using the compactness of $[0,1]$ and the continuity of $\widehat{\mu_{N'}}(\xi)$, there exists $\varepsilon>0$ such that
$$
|\widehat{\mu_{N'}}(\xi+k_\xi)|\geq \varepsilon.
$$
Finally, we complete the proof from
$$
|\widehat{\mu}(\xi+k_\xi)|=|\widehat{\mu_{N'}}(\xi+k_\xi)||\widehat{\mu_{>N'}}(\xi+k_\xi)|\geq \frac12\varepsilon.
$$
\end{proof}

\subsection{\textit{A useful estimate}\label{sect3.3}}
\

\begin{lem}\label{thm3.9}
Let integer $|p_n|>1$ and finite digit set $D_n\subset\mathbb{Z}$ for all $n\geq1$.
Suppose that $C:=\sup_n\{|p_n^{-1}d|:d\in D_n\}<\infty$.
For any $n\geq1$ and $m>0$, there exists $J_0:=\max\{[\log_2 8Cm],0\}+1$ such that $|p_{n+1}^{-1}p_{n+2}^{-1}\cdots p_{n+J_0}^{-1}|Cm<\frac14$ and $$\prod_{j=1}^\infty|\widehat{\delta_{p_{n+1}^{-1}\cdots p_{n+J_0+j}^{-1}D_{n+J_0+j}}}(\xi)|\geq\frac{2}{\pi}
$$
for any $\xi\in\mathbb{R}$ with $|\xi|<m$.
\end{lem}
\begin{proof}
Since $|p_n^{-1}|<\frac12$, there exists $0<r\leq\frac18$ such that $|p_{n+1}^{-1}\cdots p_{n+J_0}^{-1}|Cm\leq r$, i.e.,
$$| p_{n+1}^{-1}\cdots p_{n+J_0+j}^{-1}d_{n+J_0+j}\xi|\leq 2^{-j+1}r,\;\;\forall |\xi|<m$$ for any $d_{n+J_0+j}\in D_{n+J_0+j}$ and $j\geq1$.
Therefore the argument of $e^{2\pi i  p_{n+1}^{-1}\cdots p_{n+J_0+j}^{-1}d_{n+J_0+j} \xi}$ belongs to $(-2^{-j+2}r\pi,2^{-j+2}r\pi)\subset(-\frac{\pi}{2},\frac{\pi}{2})$ for any $|\xi|<m$.
Let $$A_j:=\{x+yi\in\mathbb{C}:x^2+y^2\leq1,x\geq\cos 2^{-j+2}r\pi\}.
$$
Note that $A_j$ is a convex set.
Thus the convex combination
$$
\frac{1}{\#D_{n+J_0+j}}\sum_{d\in D_{n+J_0+j}}e^{2\pi i p_{n+1}^{-1}\cdots p_{n+J_0+j}^{-1}d_{n+J_0+j} \xi}
$$ falls into $A_j$,
which yields that
\begin{align*}
|\widehat{\delta_{p_{n+1}^{-1}\cdots p_{n+J_0+j}^{-1}D_{n+J_0+j}}}(\xi)|&=\left|\frac{1}{\#D_{n+J_0+j}}\sum_{d\in D_{n+J_0+j}}e^{2\pi i p_{n+1}^{-1}\cdots p_{n+J_0+j}^{-1}d_{n+J_0+j} \xi}\right|
\\&\geq \text{dist}(\mathbf{0},A_j)
\geq\cos 2^{-j-1}\pi.
\end{align*}
Hence $$\prod_{j=1}^\infty|\widehat{\delta_{p_{n+1}^{-1}\cdots p_{n+J_0+j}^{-1}D_{n+J_0+j}}}(\xi)|\geq \prod_{j=1}^\infty\cos 2^{-j-1}\pi=\frac{2}{\pi}.$$
\end{proof}

\bigskip
\section{\bf Proof of Theorem \ref{thm2.1}\label{sect.4}}
\setcounter{equation}{0}

In Section \ref{sect.3}, we show that (ii) and (iii) of Theorem \ref{thm2.1} are equivalent.
In this section, we complete the proof of Theorem \ref{thm2.1}.
The main difficulty lies in proving (ii) $\Rightarrow$ (i). We divide it into two cases.

\subsection{\emph{(ii) $\Rightarrow$ (i) when } \textit{${\sup_n\{d:d\in \mathcal{D}_n\}=\infty}$}\label{sect4.1}}
\

We apply Wiener theorem to prove this case.
The following analyticity of Fourier transform of Borel measure is a basic tool in this section.
\begin{lem}\label{thm4.1}
Let $\mu$ be a Borel probability measure with a compact support in $\mathbb{R}$.
Then its Fourier transform
$$\widehat{\mu}(z)=\int_{\mathbb{R}}e^{2\pi itz}d\mu(t)$$
is an analytic function in $\mathbb{C}$.
\end{lem}

The following theorem (Wiener theorem) is a valuable tool to deal with the relationship between the measure and its Fourier transform when its support is contained in $[0,1)$.
Even if the support of the measure we consider is not in this range, we can still skillfully apply this theorem.
\begin{thm}[Wiener theorem, \cite{tiaohefenxi}]\label{thm4.2}
Let $\mu$ be any complex Borel measure on $[0,1)$.
Then
$$
\sum_{x}|\mu(\{x\})|^2=\lim_{N\rightarrow\infty}\frac{1}{2N+1}\sum_{-N}^N|\widehat{\mu}(n)|^2.
$$
\end{thm}

We will apply the above results to prove the following conclusion.
The main idea of the following proof goes back at least as far as Theorem 5.4 in \cite{AFL19}. And we will use some knowledge of complex analysis to get a broader conclusion.

\begin{thm}\label{thm4.3}
Let $\{(p_n, \mathcal{D}_n, L_n)\}$ be a sequence of Hadamard triples with $\sup_n\{p_n^{-1}d_n:d_n\in \mathcal{D}_n\}<\infty$ on $\mathbb{R}$, and $\nu_n$ be given as \eqref{(2.3)}.
Suppose that there exists a subsequence $\{n_k\}$ such that $\nu_{n_k+1}\xrightarrow{w}\nu$, as $k\rightarrow\infty$ and $\sup_k\#\mathcal{D}_{n_k+1}:=N<\infty$.

Then $m_0:=\#\mathcal{Z}(\nu)<\infty$, where $\mathcal{Z}(\cdot)$ is given as Proposition \ref{thm3.6}.
If $\inf_{k}\{p_{n_k+1}\}>m_0N$,
then $\{\nu_{n_k}\}$ is an equi-positive sequence.
In particular, $\mu_{\{p_n,\mathcal{D}_n\}}$ is a spectral measure.
\end{thm}
\begin{proof}
Let $K$ be a compact set with $\text{spt}\nu_{n}\subset K$ for all $n\geq1$.
It follows from Lemma \ref{thm3.5} that
$$
1=\underline{\lim}_{k\rightarrow\infty}\nu_{n_k+1}(\mathbb{R})\geq \nu(\mathbb{R})\geq\nu(K)\geq\overline{\lim}_{k\rightarrow\infty}\nu_{n_k+1}(K)=1.
$$
This implies that $\nu$ is a probability measure.
Hence, Lemma \ref{thm4.1} gives that $\widehat{\nu}(z)$ is analytic in $\mathbb{C}$ and $\#\{z\in\overline{\mathbb{D}}:\widehat{\nu}(z)=0\}<\infty$,
where $\mathbb{D}:=\{z\in\mathbb{C}:|z|<1\}$ is a unit disk.
So $m_0<\infty$.

Without loss of generality, we assume that $m_0\geq1$.
For any small $\delta>0$, using Proposition \ref{thm3.6}, we can find a small constant $\epsilon_\delta>0$ and an integer $k_x$ such that
$$
|\widehat{\nu}(x+k_x)|\geq\epsilon_\delta,\;\;\forall \;x\in[0,1]\setminus(\cup_{z\in\mathcal{Z}(\nu)}B(z;\delta)):=[0,1]\setminus B:=A.
$$
By $(iv)$ of Lemma \ref{thm3.5}, for any $p_{n_k+1}^{-1}x\in A$, there exist $K_0$ and $k=k_{p_{n_k+1}^{-1}x}$ such that
\begin{equation}\label{(4.1)}
|\widehat{\nu_{n_k}}(x+p_{n_k+1}k)|=
|\widehat{\delta_{p_{n_k+1}^{-1}\mathcal{D}_{n_k+1}}}(x)||\widehat{\nu_{n_k+1}}(p_{n_k+1}^{-1}x+k)|
\geq\frac{1}{2}|\widehat{\delta_{p_{n_k+1}^{-1}\mathcal{D}_{n_k+1}}}(x)|\epsilon_\delta,
\end{equation}
whenever $k\geq K_0$.

In the following, our goal is to find suitable integer $\ell$ such that $\frac{x+\ell}{p_{n_k+1}}\in A$ and $$\inf_{x\in[0,1)}|\widehat{\delta_{p_{n_k+1}^{-1}\mathcal{D}_{n_k+1}}}(x+\ell)|>0$$ by Wiener theorem.
Because $(p_{n_k+1},\mathcal{D}_{n_k+1}, L_{n_k+1})$ is a Hadamard triple, we get a new digit set
$$
\mathcal{D}_{n_k+1}'=\mathcal{D}_{n_k+1}\pmod{p_{n_k+1}}
$$
satisfying $\mathcal{D}_{n_k+1}'\subset\{0,1,\cdots,p_{n_k+1}-1\}$.
Write $m=\#\mathcal{D}_{n_k+1}$ and the elements of $\mathcal{D}_{n_k+1}'$ as $\{d_{n_k+1,j}'\}_{j=1}^{m}$.
Denote a discrete measure
$$
\omega_k(x)=\begin{cases}
\frac1m, & x=p_{n_k+1}^{-1}d_{n_k+1,j}';
\\
0, & \text{otherwise}.
\end{cases}
$$
For any $\xi\in[0,1)$, define the complex measure
$$
\omega_k'(E):=\int_E f(x)d\omega_k(x),
$$
where
$$
f(x)=\begin{cases}
e^{2\pi i p_{n_k+1}^{-1}d_{n_k+1,j}\xi}, & x=p_{n_k+1}^{-1}d_{n_k+1,j}'
\\
0, & \text{otherwise}.
\end{cases}
$$
Hence
$$
\widehat{\omega_k'}(y)=\frac1m\sum_{j=1}^{m}e^{2\pi ip_{n_k+1}^{-1}(yd_{n_k+1,j}'+d_{n_k+1,j}\xi)}=\widehat{\delta_{p_{n_k+1}^{-1}\mathcal{D}_{n_k+1}}}(y+\xi)
$$
whenever $y\in\mathbb{Z}$.
Since $\omega'_k$ is a complex Borel measure on $[0,1)$, applying Wiener theorem, one has
\begin{align*}
\frac{1}{\#\mathcal{D}_{n_k+1}}=\sum_{d\in \mathcal{D}_{n_k+1}'}|\omega_k'(\{\frac{d}{p_{n_k+1}}\})|^2&=
\lim_{\ell\rightarrow\infty}\frac{1}{2\ell p_{n_k+1}+1}\sum_{l=-\ell p_{n_k+1}}^{\ell p_{n_k+1}}|\widehat{\omega_k'}(l)|^2
\\
&=\frac{1}{p_{n_k+1}}\sum_{l=0}^{p_{n_k+1}-1}|\widehat{\delta_{p_{n_k+1}^{-1}\mathcal{D}_{n_k+1}}}(\xi+l)|^2.
\end{align*}
Divide the above formula into
\begin{align*}
\frac{1}{\#\mathcal{D}_{n_k+1}}&=\frac{1}{p_{n_k+1}}\sum_{\frac{\xi+l}{p_{n_k+1}}\in B}|\widehat{\delta_{p_{n_k+1}^{-1}\mathcal{D}_{n_k+1}}}(\xi+l)|^2+\frac{1}{p_{n_k+1}}\sum_{\frac{\xi+l}{p_{n_k+1}}
\in A}|\widehat{\delta_{p_{n_k+1}^{-1}\mathcal{D}_{n_k+1}}}(\xi+l)|^2
\\
&\leq \frac{1}{p_{n_k+1}}\#\{l:\frac{\xi+l}{p_{n_k+1}}\in B\}+\max_{\frac{\xi+l}{p_{n_k+1}}\in A}|\widehat{\delta_{p_{n_k+1}^{-1}\mathcal{D}_{n_k+1}}}(\xi+l)|^2
\\
&\leq 2m_0\delta+\frac{m_0}{p_{n_k+1}}+\max_{\frac{\xi+l}{p_{n_k+1}}\in A}|\widehat{\delta_{p_{n_k+1}^{-1}\mathcal{D}_{n_k+1}}}(\xi+l)|^2.
\end{align*}
Taking $\delta\leq\frac{1}{4m_0}(\frac{1}{N}-\frac{m_0}{p_{n_k}+1})$, then
$$
\max_{\frac{\xi+l}{p_{n_k}}\in A}|\widehat{\delta_{p_{n_k}^{-1}\mathcal{D}_{n_k}}}(\xi+l)|^2\geq2m_0\delta.
$$
It follows from \eqref{(4.1)} that
$$
|\widehat{\nu_{n_k}}(\xi+\ell)|\geq m_0\delta\epsilon_\delta
$$
for some $0\leq \ell\leq p_{n_k+1}-1$.
Hence $\{\nu_{n_k}\}$ is an equi-positive sequence, and $\mu_{\{p_n,\mathcal{D}_n\}}$ is a spectral measure by Theorem \ref{thmxb}.
\end{proof}
If we add the condition $\sup_n\#\mathcal{D}_n<\infty$ to the above theorem, from Theorem \ref{thm3.4} (weak compactness theorem), then we get the following corollary(by choosing $p_{n_k+1}\rightarrow\infty$, as $k\rightarrow\infty$).
\begin{cor}\label{thm4.4}
Let $\{(p_n, \mathcal{D}_n, L_n)\}$ be a sequence of Hadamard triples with $\sup_n\{p_n^{-1}d_n:d_n\in \mathcal{D}_n\}<\infty$ and $\sup_n\#\mathcal{D}_n<\infty.$
If $\sup_n\{p_n\}=\infty$,
then $\mu_{\{p_n,\mathcal{D}_n\}}$ is a spectral measure.
\end{cor}

\subsection{\emph{(ii) $\Rightarrow$ (i) when } \textit{${\sup_n\{d:d\in \mathcal{D}_n\}<\infty}$}\label{sect4.2}}
\

In this subsection, by using the diagonal principle, we will find the weak limit of $\{\nu_n\}$. And, by using some methods and techniques of complex analysis, we will prove our main results.

The following theorem is the Uniqueness part.
\begin{thm}\label{thm4.5}(Uniqueness Theorem)\cite{pro}
The Fourier transform $\widehat{\mu}$ of a probability measure $\mu$ on $\mathbb{R}^n$ characterizes $\mu$. That is, if two probabilities on $\mathbb{R}^n$ admit the same Fourier transform, then they are the same.
\end{thm}

According to Argument Principle, we conclude the following simple but practical result.
Let $\mathbb{D}=\{z\in\mathbb{C}:|z|<1\}$ be the unit disc.
\begin{lem}\label{thm4.6}
Let $\{f_n(z)\}$ be a sequence of complex analytic functions in $\mathbb{C}$, and suppose that, as $n\rightarrow\infty$, $f_n(z)\rightarrow f(z)$ uniformly on each compact subset of $\mathbb{C}$.

Denote $\mathbf{z}(f):=\{z\in\overline{\mathbb{D}}:f(z)=0\}$.
If $\#\mathbf{z}(f_n)\rightarrow\infty$ as $n\rightarrow\infty$, then $f(z)\equiv0$.
\end{lem}

\begin{proof}
If $f\not\equiv0$, then we can assume that $\mathbf{z}(f)=\{z_1,z_2,\cdots,z_m\}\subset\overline{\mathbb{D}}$.

Note that $f_n(z)\rightarrow f(z)$ uniformly on $\overline{\mathbb{D}}$ as $n\rightarrow\infty$. Choose disjoint small balls $B(z_i;r):=\{y:|y-z_i|<r\}\subset\Bbb D$.
By continuity of $f(z)$, there exists $\varepsilon>0$ such that $|f(z)|\geq\varepsilon$ for $z\in\Bbb D\setminus \cup_{i=1}^m B(z_i;r)$.
Recall that $f_n(z)\rightarrow f(z)$ for any $z\in\overline{\Bbb{D}}$.
Then there exists $N$ such that $|f(z)-f_n(z)|<\frac12\varepsilon$ for any $z\in\Bbb D$ and $n\geq N$.
Therefore
$$
|f_n(z)|\geq |f(z)|-|f(z)-f_n(z)|\geq \frac12\varepsilon>0
$$
for any $z\in\Bbb D\setminus \cup_{i=1}^m B(z_i;r)$, i.e., $\mathbf{z}(f_n)\subset \cup_{i=1}^m B(z_i;r)$ for all $n\geq N$.

Assuming that $\#(\mathbf{z}(f_n)\cap B(z_{i_n};r))=\max_i\{\#(\mathbf{z}(f_n)\cap B(z_{i};r))\}$ for some $1\leq i_n\leq m$, there exists $i_0\in\{1,2,\cdots,m\}$ such that $\#\{n\geq1:i_n=i_0\}=\infty$.
Hence, we can choose a subsequence $\{n_k\}$ such that
$$\#(\mathbf{z}(f_{n_k})\cap B(z_{i_0};r))\geq\#\mathbf{z}(f_{n_k})/m\rightarrow\infty,\;\text{as}\;k\rightarrow\infty.$$
Let $\mathbf{z}'(f_{n_k})=\mathbf{z}(f_{n_k})\cap B(z_{i_0};r)$ and $C=\{z:|z-z_{i_0}|=r\}$.
From Argument Principle, we have
$$\begin{aligned}\#\mathbf{z}'(f)=\frac{1}{2\pi}\int_{C}\frac{f'}{f}dz=\frac{1}{2\pi}\int_{C}\lim_{k\rightarrow\infty}\frac{f_{n_k}'}{f_{n_k}}dz
=\lim_{k\rightarrow\infty}\#\mathbf{z}'(f_{n_k})=\infty,
\end{aligned}$$
a contradiction.
\end{proof}

Recall that
$$
\mu_{\{p_n,\mathcal{D}_n\}}:=\delta_{p_1^{-1}\mathcal{D}_1}\ast\delta_{(p_1p_2)^{-1}\mathcal{D}_2}\ast\cdots
$$
and
$$
\nu_n=\delta_{p_{n+1}^{-1}\mathcal{D}_{n+1}}\ast\delta_{(p_{n+2}p_{n+1})^{-1}\mathcal{D}_{n+2}}\ast\cdots.
$$
We will use the above lemma to obtain the following conclusion.

\begin{thm}\label{thm4.7}
As the above notions, let $\{(p_n, \mathcal{D}_n, L_n)\}$ be a sequence of Hadamard triples with $\sup_n\{p_n^{-1}d_n:d_n\in \mathcal{D}_n\}<\infty$.
If the associated Cantor-Moran measure $\mu_{\{p_n,\mathcal{D}_n\}}$ satisfies $\limsup_{k\rightarrow\infty}T_{p_k,\mathcal{D}_k}\supset(0,1)$, then $\mathcal{Z}(\nu_n)=\emptyset$ for any $n\geq1$.

If there exists a subsequence $\{n_k\}$ such that $\nu_{n_k}\xrightarrow{w}\nu$ with $\mathcal{Z}(\nu)=\emptyset$,
then the sequence is also an equi-positive sequence. In particular, $\mu_{\{p_n,\mathcal{D}_n\}}$ is a spectral measure.
\end{thm}
\begin{proof}
For any $n\geq1$, let $x_0\in\mathcal{Z}(\nu_n)$. It implies that $x_0\in(0,1)$ and
\begin{equation}\label{(4.2)}
\widehat{\nu_n}(x_0+k)=\prod_{j=1}^\infty \widehat{\delta_{\mathcal{D}_{n+j}}}\left(\frac{x_0+k}{p_{n+1}\cdots p_{n+j}}\right)=0
\end{equation}
 for all $k\in\mathbb{Z}$.
For each $m\geq1$, let $\tau^{(m)}_{\ell}(x)=\frac{x+\ell}{p_{m}}$ for $\ell\in L_{m}^*$, where $L_{m}^*=\{0,1,\cdots,p_m-1\}$.
We call $m$ is possible for $\mathcal{S}\subset[0,1)$ if there exists $x\in\mathcal{S}$ such that there exist
$\ell_{1}\neq\ell_2\in L_{m}^*$ satisfying
$$
\widehat{\delta_{\mathcal{D}_{m}}}(\tau^{(m)}_{\ell_{i}}(x))\neq0,\;\;\: i=1,2.
$$
According to the definition of \textit{DPCS}, we know $\mathcal{S}\subset T_{p_m,\mathcal{D}_m}$ when $m$ is possible for $\mathcal{S}$.

If $n+1$ is possible for $\{x_0\}:=B_0$, then there exist two distant integers $\ell_1,\ell_2\in L_{n+1}^*$ such that $\widehat{\delta_{\mathcal{D}_{n+1}}}(\tau^{(n+1)}_{\ell_i}(x_0))\neq0$ for $i=1,2$.
In this case, denote the words set
$$
A_1:=\{1,2\}\;\text{and points set}\;B_1:=\{\tau^{(n+1)}_{\ell_{i_1}}(x_0):i_1\in A_1\}=\{x_{i_1}\}_{i_1\in A_1}.
$$
Since $x_0\in(0,1)$ and $\ell_1,\ell_2\in\{0,1,\cdots,p_{n+1}-1\}$, one has $x_1\neq x_2$ and $x_1,x_2\in(0,1)$.
Substituting $k=\ell_{i_1}+p_{n+1}k'$ into $\eqref{(4.2)}$, where $i_1\in A_1$ and $k'\in\mathbb{Z}$, we have
\begin{equation}\label{(4.3)}
\widehat{\nu_{n+1}}(\tau^{(n+1)}_{\ell_{i_1}}(x_0)+k')=\prod_{j=2}^\infty \widehat{\delta_{\mathcal{D}_{n+j}}}\left(\frac{\tau^{(n+1)}_{\ell_{i_1}}(x_0)+k'}{p_{n+2}\cdots p_{n+j}}\right)=0
\end{equation}
for all $k'\in\mathbb{Z}.$
Hence $B_1\subset\mathcal{Z}(\nu_{n+1})$.
This implies that
\begin{equation}\label{(4.4)}
\#\mathcal{Z}(\nu_{n+1})\geq\# B_1=2.
\end{equation}
If $n+1$ is not possible for $B_0$, then it implies that for any $x\in B_0$, there exists $\ell_1\in L_{n+1}$ such that $|\widehat{\delta_{D_{n+1}}}(\frac{x+\ell_1}{p_{n+1}})|=1$, as
$$
\sum_{\ell \in L_{n+1}}\left|\widehat{\delta_{\mathcal{D}_{n+1}}}\left(\frac{x+\ell}{p_{n+1}}\right)\right|^2=1.
$$
Let
$$
A_1:=\{1\}\;\text{and}\;B_1:=\{\tau^{(n+1)}_{\ell_{i_1}}(x_0):i_1\in A_1\}=\{x_{i_1}\}_{i_1\in A_1}.
$$
In this case, $\#B_1=\#B_0$.
Similarly, if $n+2$ is possible for $B_1$, by the above arguments, then there exist $\ell_{i_11}\neq\ell_{i_12}\in L_{n+2}^*$ for some $i_1\in A_1$ such that
$\widehat{\delta_{\mathcal{D}_{n+2}}}(\tau^{(n+2)}_{\ell_{l_{i_1i_2}}}(x_{i_1}))\neq0$ for $i_2=1,2$.
If $n+2$ is not possible for $B_1$, we can choose $\ell_{i_11}\in L_{n+2}$ such that
$|\widehat{\delta_{\mathcal{D}_{n+1}}}(\frac{x_{i_1}+\ell_{i_11}}{p_{n+2}})|=1$.
Let
$$
A_{2}:=\{i_1i_2:i_1\in A_1,i_2\in\{1,2\}(\text{if\;exist})\}
$$
and
$$
B_2:=\{\tau^{(n+2)}_{\ell_{i_1i_2}}(x_{i_1}):i_1i_2\in A_2\}=\{x_{i_1i_2}\}_{i_1i_2\in A_2}.
$$
Note that
$$
\tau^{(n+2)}_{l_2}\circ\tau^{(n+1)}_{l_1}(x_0)\neq\tau^{(n+2)}_{l_2'}\circ\tau^{(n+1)}_{l_1'}(x_0)
$$
for any $(l_1,l_2)\neq(l_1',l_2')\in L_{n+1}^*\times L_{n+2}^*$.
Hence $$
\#B_2\geq\begin{cases}
\# B_1, \;&n+2\;\text{is not possible for}\;B_1;
\\
\# B_1+1, \;&n+2\;\text{is possible for}\;B_1.
\end{cases}
$$
By taking $k=\ell_{i_1}+p_{n+1}\ell_{i_1i_2}+p_{n+1}p_{n+2}k'$ with $(i_1,i_1i_2)\in A_1\times A_2$ and $k'\in\mathbb{Z}$ into $\eqref{(4.2)}$, similar to \eqref{(4.3)}, we obtain that $B_2\subset\mathcal{Z}(\nu_{n+2})$, i.e.,
$$
\#\mathcal{Z}(\nu_{n+2})\geq\# B_2\geq1+\#\{1\leq i\leq 2:n+i\;\text{is possible for}\;B_{i-1}\}.
$$
Then we can define inductively the sets by
$$
A_{m}:=\{i_1\cdots i_{m}:i_1i_2\cdots i_{m-1}\in A_{m-1},\;\:\: i_m=1,2(\text{if\;exist\;})\}
$$
and
$$
B_m:=\{\tau^{(n+m)}_{\ell_{i_1i_2\cdots i_m}} (x_{i_1i_2\cdots i_{m-1}}):i_1i_2\cdots i_m\in A_m\}:=\{x_{i_1i_2\cdots i_m}\}_{i_1i_2\cdots i_m\in A_m},
$$
which satisfies $B_m\subset\mathcal{Z}(\nu_{n+m})$ and
$$
\#\mathcal{Z}(\nu_{n+m})\geq\#B_m\geq 1+\#\{1\leq i\leq m:n+i\;\text{is possible for}\;B_{i-1}\}.
$$

Suppose that $\sup_m\#B_m=\infty$, which yields $\#\mathcal{Z}(\nu_{n_k})\rightarrow\infty$, as $k\rightarrow\infty$.
According to Theorem \ref{thm3.4} (weak compactness theorem), there exists a subsequence $\nu_{n_k}$ which weakly converges to $\nu$.
It follows that $\widehat{\nu}_{n_k}(z)\rightarrow \widehat{\nu}(z)$ as $k\rightarrow\infty$, uniformly on each compact subset of $\mathbb{C}$.
In fact, for any $z_j=x_j+y_ji$ with reals $|x_j|,|y_j|\leq M$, we know that
\begin{align*}
|\widehat{\mu}(z_1)-\widehat{\mu}(z_2)|&=\left|\int e^{-2\pi \xi y_1+2\pi i\xi x_1}(1-e^{-2\pi \xi (y_2-y_1)+2\pi i\xi (x_2-x_1)}d\nu_{n_k}\right|
\\
&\leq e^{2\pi KM}(|1-e^{2\pi K|y_2-y_1|}|+e^{2\pi  K}2\pi K|x_2-x_1|),
\end{align*}
where $K=\max\{|x|,1:x\in\text{spt}(\nu_n)\cup\text{spt}(\nu)$ and $\mu\in\{\nu_n\}\cup\{\nu\}$.
It is easy to see that $\{\widehat{\nu_{n}}(z)\}\cup\{\widehat{\nu}(z)\}$ is equicontinuous for $z\in\Bbb C$ with $|z|\leq M$.
Note that
$$\sup\{|\widehat{\mu}(z)|:\mu\in\{\nu_{m_k}\}\cup\{\nu\},|z|\leq M\}<\infty.
$$
and $\widehat{\nu_{m_k}}(z)\rightarrow\widehat{\nu}(z)$ for any $z\in\Bbb C$.
Applying Arzel\`a-Ascoli Theorem, we get $\widehat{\nu_{m_k}}(z)$ converges uniformly to $\widehat{\nu}(z)$ on $|z|\leq M$.
By Lemma \ref{thm4.6}, we have $\widehat{\nu}(x)\equiv0$, i.e., $\nu=0$, a contraction.
Hence, there exists $N_0\geq1$ such that $\#B_{n+N_0}=\#B_{n+j}$ for all $j\geq N_0$.
In particular, from Lemma \ref{thm3.2}, if $p_{n+j}>\#\mathcal{D}_{n+j}$, then $(p_{n+j},\mathcal{D}_{n+j})$ satisfies \textit{DPC}.
Thus $n+j$ is possible for any set of single point $\{x\}\subset[0,1)$.
This implies that
$$
\#\{i\geq1:p_{n+j}>\#\mathcal{D}_{n+j}\}<\infty.
$$

By the above arguments, we have $p_{n+j}=\#\mathcal{D}_{n+j}$ for all $j\geq N_0$.
It follows from Lemma \ref{thm3.3} that $T_{p_{n+j},\mathcal{D}_{n+j}}=[0,1)\setminus\frac{1}{\gcd \mathcal{D}_{n+j}}\mathbb{Z}$.
We know $n+j$ is not possible for $B_{j-1}$ for all $j\geq N_0+1$.
Therefore $B_{j-1}\subset \frac{1}{\gcd \mathcal{D}_{n+j}}\mathbb{Z}$ for all $j\geq N_0+1$.
Fixing $j\geq N_0+1$, put $\frac{p}{q}=\frac{v_{n+j}}{\gcd\mathcal{D}_{n+j}}=x_{i_1i_2\cdots i_{j-1}}\in B_{j-1}$, where $v_{n+j}\in\mathbb{Z}$, $\gcd(p,q)=1$ and $q>1$.
The construction of $B_{j}$ allows us to find $i_1i_2\cdots i_{j-1}1\in A_{n+j}$ such that $$x_{i_1i_2\cdots i_{j-1}1}=\frac{x_{i_1i_2\cdots i_{j-1}}+\ell_{i_1i_2\cdots i_{j-1}1}}{p_{n+j}}\in B_{j}\subset\frac{1}{\gcd\mathcal{D}_{n+j+1}}\mathbb{Z}.$$
Hence $q|\gcd\mathcal{D}_{n+j+1}$.
Similarly, it follows that
$$
\frac{x_{i_1i_2\cdots i_{j-1}1}+\ell_{i_1i_2\cdots i_{j-1}11}}{p_{n+j+1}}\in B_{j+1}\subset\frac{1}{\gcd \mathcal{D}_{n+j+2}}\mathbb{Z}
$$
for some $i_1i_2\cdots i_{j-1}11\in A_{n+j+1}$.
Then $q|\gcd\mathcal{D}_{n+j+2}$.
By induction, one has $q|\gcd\mathcal{D}_{n+j+i}$ for all $i\geq0$.
Then
$$\cup_{k\geq n+j} T_{p_k,\mathcal{D}_{k}}=[0,1)\setminus(\cap_{k\geq n+j}\frac{1}{\gcd\mathcal{D}_k}\mathbb{Z})\subset [0,1)\setminus\frac{1}{q}\mathbb{Z}.
$$
Noting the assumption $\limsup_{k\rightarrow\infty}T_{p_k,\mathcal{D}_k}\supset(0,1)$,
it yields that $q=1$, i.e. $x_0\in\mathbb{Z}$, a contradiction.
Therefore $\mathcal{Z}(\nu_n)=\emptyset$ for all $n\geq1$.

Finally, the theorem follows from Lemma \ref{thm3.7} and Theorem \ref{thmxb}.
\end{proof}
\begin{thm}\label{thm4.8}
Under the assumption of Theorem \ref{thm2.1}, if $\limsup_{k\rightarrow\infty}T_{p_k,\mathcal{D}_k}\supset(0,1)$ and $\sup_n\{p_n\}<\infty$, then there exists an equi-positive subsequence $\{\nu_{n_k}\}$.
In particular,
$\mu_{\{p_n,\mathcal{D}_n\}}$ is a spectral measure.
\end{thm}
\begin{proof}
\textbf{Step 1: Find a subsequence $\{\nu_{n_k}\}$ and the weak limit $\nu$.}

The conditions $\sup_n\{p_n:n\geq1\}<\infty$ and $\sup_n\{p_n^{-1}d:d\in\mathcal{D}_n\}<\infty$ imply that $\sup_n\{d:d\in \mathcal{D}_n\}<\infty$. Choose $m_1\in \mathbb{N}$ such that
$$(p_{m_1+1},\mathcal{D}_{m_1+1})=(p_{k+1}, \mathcal{D}_{k+1})\;\text{for\;infinitely\;many}\;k, $$
and denote
$$\mathcal{A}_1:=\{k>m_1:(p_{m_1+1},\mathcal{D}_{m_1+1})=(p_{k+1}, \mathcal{D}_{k+1})\}.$$
Similarly, choose the smallest integer $m_2\in\mathcal{A}_1$ such that
$$(p_{m_2+2},\mathcal{D}_{m_2+2})=(p_{k+2},\mathcal{D}_{k+2})\;\text{for\;infinitely\;many}\;k\in \mathcal{A}_1$$
and let
$$
\mathcal{A}_2:=\{k\in \mathcal{A}_1:(p_{m_2+2},\mathcal{D}_{m_2+2})=(p_{k+2}, \mathcal{D}_{k+2})\}.
$$
By induction, we get a sequence $\{m_k\}$ such that
$$(p_{m_k+k},\mathcal{D}_{m_k+k})=(p_{k'+k},\mathcal{D}_{k'+k})\;\text{for\;infinitely\;many}\;k'\in \mathcal{A}_{k-1}.$$
Applying another diagonalization process to $\{(p_{m_k+i},\mathcal{D}_{m_k+i})\}$, we obtain a further subsequence $\{(p_{m_k+k},\mathcal{D}_{m_k+k})\}_{k=1}^\infty$ which satisfies
$$
(p_{m_k+i},\mathcal{D}_{m_k+i})=(p_{m_{k'}+i},\mathcal{D}_{m_{k'}+i})
$$
for any $k<k'$ and $1\leq i\leq k$.
Rewrite this new sequence as
$$(p_k',\mathcal{D}_{k}',L_k'):=(p_{m_k+k},\mathcal{D}_{m_k+k},L_{m_k+k})$$ for all $k\geq1$.
Then the associated Moran measure
$$
\mu_{\{p_k',\mathcal{D}_k'\}}(\cdot)=\delta_{p_1'^{-1}\mathcal{D}_1'}\ast\delta_{p_1'^{-1}p_2'^{-1}\mathcal{D}_2'}\ast\cdots.
$$
From Theorem \ref{thm3.4}, without loss of generality, we assume that $\{\nu_{m_k}\}$ weakly converges
to $\nu$.
At the end of the step, we show the weak limit $\nu$ is $\mu_{\{p_k',\mathcal{D}_k'\}}$.
A basic fact is $\prod_{j=n+1}^\infty \cos2^{-j}\pi\rightarrow1$, as $n\rightarrow\infty$.
So for any $\epsilon>0$, there exists $N>0$ such that $|\prod_{j=n+1}^\infty \cos2^{-j}\pi-1|<\frac12\epsilon$
and $|2^{-n}\pi|\leq \frac14\epsilon$ for all $n\geq N$.
Fixed $x\in\Bbb R$, let $k>J_0+N$, where $J_0$ is given as Lemma \ref{thm3.9}.
Then
$$
|\widehat{\mu_{\{p_k',\mathcal{D}_k'\}}}(x)-\widehat{\nu_{m_k}}(x)|
=\prod_{j=1}^{k}|\widehat{\delta_{p_1^{'-1}\cdots p_j^{'-1}\mathcal{D}_j'}}(x)|\cdot
\left|\prod_{j=k+1}^{\infty}\widehat{\delta_{p_1^{'-1}\cdots p_j^{'-1}\mathcal{D}_j'}}(x)-
\prod_{j=m_k+k+1}^{\infty}\widehat{\delta_{p_{m_k+1}^{-1}\cdots p_j^{-1}\mathcal{D}_j}}(x)\right|.
$$
By the proof of Lemma \ref{thm3.9} and $k>J_0+N$, then
$$
\widehat{\delta_{p_1^{'-1}\cdots p_j^{'-1}\mathcal{D}_j'}}(x),
\widehat{\delta_{p_{m_k+1}^{-1}\cdots p_j^{-1}\mathcal{D}_j}}(x)\in A_{j-k+N}\subset\{re^{i\theta}:\cos2^{-(j-k+N)}\pi\leq r\leq 1, |\theta|\leq 2^{-(j-k+N)}\pi\}
$$
for all $j>k$.
Hence
$
\prod_{j=k+1}^{\infty}\widehat{\delta_{p_1^{'-1}\cdots p_j^{'-1}\mathcal{D}_j'}}(x),
\prod_{j=m_k+k+1}^{\infty}\widehat{\delta_{p_{m_k+1}^{-1}\cdots p_j^{-1}\mathcal{D}_j}}(x)
$
belong to
$$
A:=\left\{re^{i\theta}:1-\frac12\epsilon\leq\prod_{j=N+1}^\infty\cos2^{-j}\pi\leq r\leq 1, |\theta|\leq 2^{-N}\pi\leq \frac14\epsilon\right\}.
$$
It follows from $|\widehat{\delta_{p_1'^{-1}\cdots p_j^{'-1}}}(x)|\leq1$ that
$$
|\widehat{\mu_{\{p_k',\mathcal{D}_k'\}}}(x)-\widehat{\nu_{m_k}}(x)|
=|A|\leq |\{r:1-\frac12\epsilon\leq r\leq 1\}|+
|\{e^{i\theta}:|\theta|\leq \frac14\epsilon\}|\leq\epsilon,
$$
where $|E|$ denotes the diameter of $E$.
This implies that
$\widehat{\mu_{\{p_k',\mathcal{D}_k'\}}}(x)=\lim_{k\rightarrow\infty}\widehat{\nu_{m_k}}(x)=\widehat{\nu}(x)$ for any $x\in\mathbb{R}$.
Theorem \ref{thm4.5} (Uniqueness Theorem) gives $\mu_{\{p_k',\mathcal{D}_k'\}}=\nu$.

If the sequence $\{(p_k',\mathcal{D}_k')\}$ satisfies $\limsup_{k\rightarrow\infty}T_{p_k',\mathcal{D}_k'}\supset(0,1)$, then $\mathcal{Z}(\nu)=\emptyset$ by Theorem \ref{thm4.7}, and $\mu_{\{p_n,\mathcal{D}_n\}}$ is a spectral measure.
So we only need to consider $p_n'=\#\mathcal{D}_n'$ and $\gcd_{j\geq n}\{\gcd\mathcal{D}_j'\}\equiv d>1$ for all $n\geq1$.
We assume that $x\in\mathcal{Z}(\nu)\neq\emptyset$.

\textbf{Step 2: Determine $x$.}

Let $\iota^{(i)}_{\ell}(x):=\frac{x+\ell}{p_{i}'}, \ell\in L_i'$ for any $i\geq1$.
There exists $\ell_1\in L_i'$ such that
$|\widehat{\delta_{\mathcal{D}_1'}}(\iota^{(1)}_{\ell_1}(x))|=\max\{|\widehat{\delta_{\mathcal{D}_1'}}(\iota^{(1)}_{\ell}(x))|:\ell\in L_1'\}>0$, as
$$
\sum_{\ell\in L_1'}|\widehat{\delta_{\mathcal{D}_1'}}(\iota^{(1)}_\ell(x))|^2=1.
$$
Similarly, we can find $\ell_2\in L_i'$ such that
$$
|\widehat{\delta_{\mathcal{D}_2'}}(\iota^{(2)}_{\ell_2}\circ\iota^{(1)}_{\ell_1}(x))|=
\max\{|\widehat{\delta_{\mathcal{D}_2'}}(\iota^{(2)}_{\ell}\circ\iota^{(1)}_{\ell_1}(x))|:\ell\in L_2'\}>0.
$$
By induction, we obtain a sequence $\{\ell_j\}_{j=1}^\infty$ such that
$$
|\widehat{\delta_{\mathcal{D}_j'}}(\iota^{(j)}_{\ell_j}\circ\cdots\circ \iota^{(1)}_{\ell_1}(x))|=\max\{|\widehat{\delta_{\mathcal{D}_j'}}(\iota^{(j)}_{\ell}\circ\iota^{(j-1)}_{\ell_{j-1}}\circ\cdots\circ \iota^{(1)}_{\ell_1}(x))|:\ell\in L_j'\}>0.
$$
Let
$$
\kappa_n:=\delta_{p_1'^{-1}\mathcal{D}_1'}\ast\cdots\ast\delta_{(p_1'\cdots p_n')^{-1}\mathcal{D}_n'},\;
\kappa_{>n}:=\delta_{p_{n+1}'^{-1}\mathcal{D}_{n+1}'}\ast\delta_{p_{n+1}'p_{n+2}'^{-1}\mathcal{D}_{n+2}'}\ast\cdots
$$
for all $n\geq1$.
By the proof of Theorem \ref{thm4.7}, we have
$$
\#\mathcal{Z}(\kappa_{>j})\geq
\#\{1\leq i\leq j-1: 0<|\widehat{\delta_{\mathcal{D}_i'}}(\tau^{(i)}_{\ell_i}\circ\cdots\circ \tau^{(1)}_{\ell_1}(x))|<1\}.
$$
The relation between $\kappa_{>n}$ and $\nu_n$ tells us that $\mathscr{Z}(\widehat{\kappa_{>n}})\cap[0,1)\subset\mathscr{Z}(\widehat{\nu_n})\cap[0,1)$, i.e.,  $$\sup_n\#(\mathscr{Z}(\widehat{\kappa}_{>n})\cap[0,1))\leq k_0:=\sup_n\#(\mathscr{Z}(\widehat{\nu}_n)\cap[0,1))<\infty,$$
which implies that
\begin{equation}\label{(4.5)}
\#\{1\leq i\leq j-1: 0<|\widehat{\delta_{\mathcal{D}_i'}}(\iota^{(i)}_{\ell_i}\circ\cdots\circ \iota^{(1)}_{\ell_1}(x))|<1\}\leq k_0
\end{equation}
for all $j\geq1$.
Hence there exists an integer $s$ such that
$$
|\widehat{\delta_{\mathcal{D}_j'}}(\iota^{(j)}_{\ell_j}\circ\cdots\circ \iota^{(1)}_{\ell_1}(x))|=1
$$
for any $j\geq s$.
Similar to the proof of Theorem \ref{thm4.7}, we know
$$
x_0:=\iota^{(s)}_{\ell_{s}}\circ\iota^{(s-1)}_{\ell_{s-1}}\circ\cdots\circ \iota^{(1)}_{\ell_1}(x)\in\frac{\mathbb{Z}}{d}.
$$

\textbf{Step 3: Complete the estimate.}

Let $t:=\sup_j\#\mathcal{D}_j'$.
Note that
$$
|\widehat{\delta_{\mathcal{D}_j'}}(\iota^{(j)}_{\ell_j}\circ\cdots\circ \iota^{(1)}_{\ell_1}(x))|^2\geq\frac{1}{t}\sum_{\ell\in L_j'}|\widehat{\delta_{\mathcal{D}_j'}}(\iota^{(j)}_{\ell}\circ\cdots\circ \iota^{(1)}_{\ell_1}(x))|^2\equiv\frac{1}{t}.
$$
It follows from \eqref{(4.5)} that
$$
\prod_{j=1}^s\left|\widehat{\delta_{\mathcal{D}_j'}}\left(\frac{x+\ell_1+p_1'\ell_2+\cdots+p_1'\cdots p_{s-1}'\ell_s}{p_1'\cdots p_j'}\right)\right|^2\geq\frac{1}{t^{k_0}}.
$$
Letting $k\geq s$, we have
$$
|\widehat{\nu_{m_k}}(x+\ell_1+\cdots+p_1\cdots p_{s-1}\ell_s)|=|\widehat{\nu_{m_k+s}}(x_0)|\cdot\prod_{j=1}^s\left|\widehat{\delta_{\mathcal{D}_j'}}\left(\frac{x+\ell_1+\cdots+p_1\cdots p_{s-1}\ell_s}{p_1'\cdots p_j'}\right)\right|.
$$
We claim that there exists an integer $J$ such that
$$
\prod_{j=J+1}^\infty\left|\widehat{\delta_{\mathcal{D}_{n+j}}}\left(\frac{\xi}{\widetilde{p}_1\cdots \widetilde{p}_{j}}\right)\right|\geq\frac{2}{\pi},\;\;n\geq1
$$
for any $\widetilde{p}_i\in\{p_i\}_i$ and $|\xi|\leq1$.
In fact, by choosing $J$ satisfying $\frac{\max_n\{d:d\in \mathcal{D}_n\}}{2^J}\leq \frac14$, we have
$$
\prod_{j=J+1}^\infty\left|\widehat{\delta_{\mathcal{D}_{n+j}}}\left(\frac{\xi}{\widetilde{p}_1\cdots \widetilde{p}_{j}}\right)\right|\geq
\prod_{j=J+1}^\infty\left|\text{Re}\left(\widehat{\delta_{\mathcal{D}_{n+j}}}\left(\frac{\xi}{\widetilde{p}_1\cdots \widetilde{p}_{j}}\right)\right)\right|\geq
\prod_{j=1}^\infty \cos\frac{\pi}{2^{j+1}}=\frac{2}{\pi}.
$$
Let
$$
U:=\left\{\left|\widehat{\delta_\mathcal{D}}\left(\frac{x}{\widetilde{p}_1\cdots \widetilde{p}_{j}}\right)\right|^2\neq0:\mathcal{D}\in\{\mathcal{D}_i\}, \widetilde{p}_i\in\{p_i\},x\in\bigcup_j\frac{\mathbb{Z}}{\gcd\mathcal{D}_{j}},1\leq j\leq J+k_0\right\}
$$
and
$
\min U:=\varepsilon_0'>0.
$

If $\widehat{\nu_{m_k+s}}(x_0)\neq0$, then $$|\widehat{\nu_{m_k}}(x+\ell_1+\cdots+p_1\cdots p_{s-1}\ell_s)|^2\geq|\widehat{\nu_{m_k+s}}(x_0)|^2\frac{1}{t^{k_0}}\geq\varepsilon_0'^J\frac{4}{\pi^2 t^{k_0}}:=\varepsilon_0''>0.$$
If $\widehat{\nu_{m_{k}+s}}(x_0)=0$, by
$$
\sum_{\ell\in L_{m_k+s+1}}\left|\widehat{\delta_{\mathcal{D}_{m_k+s+1}}}\left(\frac{x_0+\ell}{p_{m_k+s+1}}\right)\right|^2\equiv1,
$$
we can choose $\ell_{s+1}\in L_{m_k+s+1}$ such that $|\widehat{\delta_{\mathcal{D}_{m_k+s+1}}}(x_1)|\neq0$, where $x_1:=\frac{x_0+\ell_{s+1}}{p_{m_k+s+1}}$.
If $|\widehat{\delta_{\mathcal{D}_{m_k+s+1}}}(x_1)|=1$, then $x_1\in\frac{\mathbb{Z}}{\gcd\mathcal{D}_{m_k+s+1}}$.
If $|\widehat{\delta_{\mathcal{D}_{m_k+s+1}}}(x_1)|<1$, then there exists $\ell_{s+1}'\in L_{m_k+s+1}\setminus\{\ell_{s+1}\}$ such that
$|\widehat{\delta_{\mathcal{D}_{m_k+s+1}}}(x_2)|\neq0$, where $x_2:=\frac{x_0+\ell_{s+1}'}{p_{m_k+s+1}}$.
Note that $x_i\in\frac{\mathbb{Z}}{dp_{m_k+s+1}}$.
So if $\widehat{\nu_{m_k+s+1}}(y)\neq0$ for some $y\in\{x_1,x_2\}$(with loss of generality, we assume $y=x_1$), then $|\widehat{\delta_{\mathcal{D}_{m_k+s+1}}}(y)|\geq\varepsilon_0'$, this yields that
$$
|\widehat{\nu}_{m_k}(x+\ell_1+\cdots+p_1\cdots p_{s}\ell_{s+1})|\geq\varepsilon_0'\varepsilon_0''.
$$
If $\widehat{\nu_{m_k+s+1}}(y)=0$ for all $y\in\{x_1,x_2\}$(or $\{x_1\}$), we write $\mathcal{C}_1=\{x_1,x_2\}\subset\mathscr{Z}(\widehat{\nu_{m_k+s+1}})$.

For any $x_i\in\mathcal{C}_1$, there exists $\ell_{s+2}\in L_{m_k+s+2}$ such that $|\widehat{\delta_{\mathcal{D}_{m_k+s+2}}}(x_{i1})|\neq0$, where $x_{i1}:=\frac{x_i+\ell_{s+2}}{p_{m_k+s+2}}$.
If $|\widehat{\delta_{\mathcal{D}_{m_k+s+2}}}(x_{i1})|=1$, then $x_{i1}\in\frac{\mathbb{Z}}{\gcd\mathcal{D}_{m_k+s+2}}$.
If $|\widehat{\delta_{\mathcal{D}_{m_k+s+2}}}(x_{i1})|<1$, then there exists $\ell_{s+2}'\in L_{m_k+s+2}\setminus\{\ell_{s+2}\}$ such that
$|\widehat{\delta_{\mathcal{D}_{m_k+s+2}}}(x_{i2})|\neq0$, where $x_{i2}:=\frac{x_i+\ell_{s+2}'}{p_{m_k+s+2}}$.
Note that $x_{ij}\in\frac{\mathbb{Z}}{\gcd\mathcal{D}_{m_k+s+1}p_{m_k+s+1}p_{m_k+s+2}}$.
So if $\widehat{\nu_{m_k+s+2}}(y)\neq0$ for some $y\in\{x_{i1},x_{i2}\}$, then $|\widehat{\delta_{\mathcal{D}_{m_k+s+1}}}(y)|\geq\varepsilon_0'$, this yields that
$$
|\widehat{\nu}_{m_k}(x+\ell_1+\cdots+p_1\cdots p_{s+1}\ell_{s+2})|\geq\varepsilon_0'^2\varepsilon_0''.
$$
If $\widehat{\nu_{m_k+s+2}}(y)=0$ for all $y\in\{x_{i1},x_{i2}\}_{x_i\in\mathcal{C}_1}$(if exist), we write $\mathcal{C}_2=\{x_{ij}:x_i\in\mathcal{C}_1\}$.
It is easy to check that
$\mathcal{C}_2\subset\mathscr{Z}(\widehat{\nu_{m_k+s+2}})$ and $\#\mathcal{C}_2\geq \#\mathcal{C}_1+\#\{x_i\in\mathcal{C}_1:0<|\widehat{\delta_{\mathcal{D}_{m_k+s+2}}}(x_{i1})|<1\}$.
In fact, the construct of $\mathcal{C}_2$ gives $x_{ij}\neq x_{i'j'}$ for any $ij\neq i'j'$.

By induction, there exists $\Omega_l\subset\{1,2\}^l$ such that the sequence $\mathcal{C}_l:=\{x_i\}_{i\in\Omega_l}\subset\mathscr{Z}(\widehat{\nu_{m_k+s+l}})$ and
$$
\#\mathcal{C}_l\geq \#\mathcal{C}_{l-1}+\#\{x_i\in\mathcal{C}_{l-1}:0<|\widehat{\delta_{\mathcal{D}_{m_k+s+l}}}(x_{i1})|<1\}.
$$

Recall that $\#\mathscr{Z}(\widehat{\nu_{m_k+s+l}}\cap[0,1))\leq k_0$ for all $l\geq1$, we know $\#\mathcal{C}_l\leq k_0$.
If $l=\infty$, then there exists an integer $j_0$ such that $|\widehat{\delta_{\mathcal{D}_{m_k+s+j}}}(x_i)|=1$ for all $i\in\Omega_j$ with $j\geq j_0$.
This yields that
$$
x_i\in\frac{\mathbb{Z}}{\gcd(\gcd\mathcal{D}_{m_k+s+j_0}, \gcd\mathcal{D}_{m_k+s+j})},\;\;\forall i\in\Omega_j,\;j\geq j_0.
$$
According to $\gcd_{j\geq n}\{\gcd \mathcal{D}_j\}\equiv1,\;\;\forall n\geq1$, we can find an integer $l'$ such that $x_{i}\in\mathbb{Z}$ for all $i\in\Omega_{l'}$, which is impossible, $i.e.$, $l<\infty$.
Let $k_{x,\nu_{m_k}}=\ell_1+\cdots+p_1\cdots p_{s+l-1}\ell_{s+l}$, it follows that
$$
|\widehat{\nu_{m_k}}(x+k_{x,\nu_{m_k}})|^2\geq{\varepsilon_0'^{k_0}} \varepsilon_0'':=\varepsilon_0'''>0.
$$
By Lemma 2.2 of \cite{AFL19}, we know that $\{\widehat{\nu_{m_k}}\}$ is equicontinuous.
This together with the above estimate implies that
\begin{equation}\label{(4.6)}
|\widehat{\nu_{m_k}}(y+k_{x,\nu_{m_k}})|^2\geq\frac{1}{2} \varepsilon_0'''>0,
\end{equation}
whenever $\text{dist}(y,\mathcal{Z}(\nu))<\delta_0$ for some $\delta_0$.
Let $E:=[0,1]\setminus\{y\in(0,1):\text{dist}(y,\mathcal{Z}(\nu))<\delta_0\}$.
Lemma \ref{thm3.6} implies that there exist $\widetilde{\varepsilon}$ and $K_0'$ such that
\begin{equation}\label{(4.7)}
|\widehat{\nu}_{m_k}(x+k_{x})|\geq\widetilde{\varepsilon},\;\:\: k\geq K_0'
\end{equation} for any $x\in E$.
Combining \eqref{(4.6)} and \eqref{(4.7)}, $\{\nu_{m_k}\}$ is an equi-positive sequence.
Hence $\mu_{\{p_n,\mathcal{D}_n\}}$ is a spectral measure.
\end{proof}

\subsection{\emph{Completion of the proof of Theorem \ref{thm2.1}}}\label{sect4.3}
\

\begin{proof}[\textbf{Proof of Theorem \ref{thm2.1}}]
(ii) $\Leftrightarrow$ (iii) holds at the end of Section \ref{sect.3}.
Corollary \ref{thm4.4} and Theorem \ref{thm4.8} imply that (ii) $\Rightarrow$ (i).

In the following, we complete the proof by showing (i) $\Rightarrow$ (ii).
If (ii) doesn't hold, by (ii) $\Leftrightarrow$ (iii), we assume that $p_n=\#\mathcal{D}_n$ and $d_n'=\gcd_{j\geq n}\gcd \mathcal{D}_j>1$ for all large $n$.
Now, we show $\mathcal{Z}(\nu_n)\neq\emptyset$.

Choose $\xi_0=\frac{1}{d_n'}$.
For $j\geq n$, let $\mathcal{D}_j=d_n'\mathcal{D}_j'$ for some $\mathcal{D}_j'\subset\mathbb{Z}$.
For any integer $k\geq0$, write $d_n'k=k_1+k_2p_{n+1}+k_3p_{n+1}p_{n+2}+\cdots$, where $k_j\in\{0,1,\cdots,p_j-1\}$ and $k_j\equiv0$ for all large $j$.
Denote $J=\min\{j:k_j\neq p_j-1\}$.
By the proof of Lemma \ref{thm3.3}, we have $\gcd(p_{n+j},\gcd\mathcal{D}_{n+j})=1=\gcd(p_{n+j},d_n')$.
Note that $(p_{n+j},\mathcal{D}_{n+j},L_{n+j})$ forms a Hadamard triple and $p_{n+j}=\#\mathcal{D}_{n+j}$.
Then
$\mathcal{D}_{n+j}'=\mathcal{D}_{n+j}=\{0,1,\cdots,p_{n+j}-1\}\pmod {p_{n+j}}$.
Hence
\begin{equation}\label{(4.8)}
\widehat{\delta_{\mathcal{D}_{n+j}'}}(\frac{k}{p_{n+j}})=\frac{1}{\#\mathcal{D}_{n+j}'}\sum_{d\in\mathcal{D}_{n+j}'}e^{2\pi i\frac{dk}{p_{n+j}}}
=\frac{1}{\#\mathcal{D}_{n+j}'}\sum_{d=0}^{p_{n+j}-1}e^{2\pi i\frac{dk}{p_{n+j}}}=0
\end{equation}
for any $k\in\{1,2,\cdots,p_{n+j}-1\}$.
It follows from \eqref{(4.8)} that
\begin{align*}
\widehat{\nu_{n}}(\xi_0+k)&=\prod_{j=1}^\infty \widehat{\delta_{\mathcal{D}_{n+j}}}((p_{n+1}\cdots p_{n+j})^{-1}(\xi_0+k))
\\
&=\prod_{j=1}^\infty \widehat{\delta_{\mathcal{D}_{n+j}'}}((p_{n+1}\cdots p_{n+j})^{-1}(1+d_n'k))
\\
&=\prod_{j=J}^\infty \widehat{\delta_{\mathcal{D}_{n+j}'}}((p_{n+J}\cdots p_{n+j})^{-1}(1+k_J+p_{n+J}k_{J+1}+\cdots))
\\
&=0.
\end{align*}
Similarly, $\widehat{\nu_n}(\xi_0+k)=0$ for all negative integer $k$, i.e., $\xi_0\in\mathcal{Z}(\nu_n)$,
which implies that $\{\nu_n\}$ doesn't have equi-positive subsequence.
\end{proof}

\bigskip
\section{\bf The singularly and absolutely continuous Cantor-Moran measures\label{sect.5}}
\setcounter{equation}{0}

In this section, we consider the spectrality of the singularly and absolutely continuous Cantor-Moran measures.
We show all that all singularly continuous Cantor-Moran measures are spectral in first subsection.
Later, we study Fuglede's Conjecture on Cantor-Moran set.

\subsection{\emph{The singularly continuous Cantor-Moran spectral measure}\label{subsect5.1}}
\

In this subsection, we explore the relationship between the spectrality and singularity of Cantor-Moran measure.
In particular, we prove Theorem \ref{thm2.2}.



The following results about fractal geometry need to be known.
\begin{thm}[\cite{fenxing}]\label{thm5.1}
Let $\mu$ be a mass distribution on $\mathbb{R}^n$, let $\mathcal{H}^s$ be $s$-dimensional Hausdorff measure.
Suppose that $F\subset\mathbb{R}^n$ is a Borel set and $0<c<\infty$ is a constant.

If $\overline{\lim}_{r\rightarrow0}\mu(B(x,r))/r^s<c$ for all $x\in F$, then $\mathcal{H}^s(F)\geq \mu(F)/c$.
\end{thm}
The following lemma is an extension of Proposition 4.9 in \cite{fenxing}.
\begin{lem}\label{thm5.2}
Let $\mu$ be a finite Borel measure on $\mathbb{R}^n$. Suppose that $F\subset\mathbb{R}^n$ is a Borel set and $c(r)$ is a positive function with $c(r)\rightarrow\infty$ as $r\rightarrow0^+$.

If $\mu(B(x,r))/r^s\geq c(r)$ for all $x\in F$ and all $r>0$, then $\mathcal{H}^s(F)=0$.
\end{lem}
\begin{proof}
Suppose that $F$ is bounded.
For any large $M>0$, there exists $R>0$ such that $c(r)\geq M$ for any $0<r<R$.
Fix $\delta\in(0,R)$ and let $\mathcal{C}$ be the collection of balls
$$
\{B(x,r):x\in F, 0<r\leq \delta\;\text{and}\;\mu(B(x,r))\geq c(r)r^s\}.
$$
By the hypothesis, we have $F\subset \cup_{B\in\mathcal{C}}B$.
Applying the Covering lemma (\cite{fenxing}, Lemma 4.8) to the collection $\mathcal{C}$, there is a sequence of disjoint balls $B_i\in\mathcal{C}$ such that $\cup_{B\in\mathcal{C}}B\subset\cup_i\widetilde{B}_i$, where $\widetilde{B}_i$ is the ball concentric with $B_i$ but of
four times the radius.
Hence $\{\widetilde{B}_i\}_{i=1}^\infty$ is an $8\delta$-cover of $F$, so
$$
\mathcal{H}_{8\delta}^s(F)\leq \sum_i|\widetilde{B}_i|^s\leq 4^s\sum_{i}|B_i|^s\leq 4^s\sum_i\frac{\mu(B_i)}{c(|B_i|)}\leq \frac{4^s}{M}\mu(\mathbb{R}^n),
$$
where $|B_i|$ is the diameter of $B_i$.
Letting $\delta\rightarrow0$, then
$\mathcal{H}^s(F)\leq \frac{4^s\mu(\mathbb{R}^n)}{M}$.
Thus $\mathcal{H}^s(F)=0$ by setting $M\rightarrow\infty$.

Finally, if $F$ is unbounded and $\mathcal{H}^s(F)>0$, the $\mathcal{H}^s$-measure of some bounded
subset of $F$ will also exceed $0$, contrary to the above.
\end{proof}
\begin{thm}\label{thm5.3}
Let $\{(p_n, \mathcal{D}_n, L_n)\}$ be a sequence of Hadamard triples with $\sup_n\{p_n^{-1}d:d\in\mathcal{D}_n\}<\infty$ and $\sup_n\#\mathcal{D}_n<\infty.$
Then the associated Cantor-Moran measure
$$
\mu_{\{p_n,\mathcal{D}_n\}}=\delta_{p_1^{-1}\mathcal{D}_1}\ast\cdots\ast\delta_{p_1^{-1}p_2^{-1}\mathcal{D}_2}\ast\cdots
$$
is singular continuous with respect to the Lebesgue measure if and only if $\#\{n:p_n>\#\mathcal{D}_n\}=\infty$.
\end{thm}
\begin{proof}
Note that the support of $\mu_{\{p_n,\mathcal{D}_n\}}$ is
$$
T=\left\{\sum_{j=1}^\infty\frac{d_j}{p_1\cdots p_j}:d_j\in \mathcal{D}_j\right\}.
$$
For any $x\in T$, write
$$
x=\sum_{j=1}^\infty\frac{d_j}{p_1\cdots p_j},
$$
where $d_j\in \mathcal{D}_j$.
For any small $r>0$, let
$$
\frac{1}{p_1\cdots p_{s_r}}\leq r<\frac{1}{p_1\cdots p_{s_r-1}}
$$
for some integer $s_r\geq1$.
Put $c:=\sup_n\{\frac{d}{p_n}:d\in \mathcal{D}_n\}$, $k_0:=[\log_24c]+1$ and $x_r=\sum_{j=1}^{s_r+k_0}\frac{d_j}{p_1\cdots p_j}$.
Note that
$$
|x-x_r|\leq\sum_{j=s_r+k_0+1}^{\infty}\frac{d_j}{p_1\cdots p_j}\leq \frac{2c}{p_1\cdots p_{s_r+k_0}}\leq\frac{2c}{2^{k_0}p_1\cdots p_{s_r}}\leq\frac{1}{2p_1\cdots p_{s_r}}\leq \frac{1}{2}r.
$$
Thus
\begin{equation}\label{(5.1)}
\mu_{\{p_n,\mathcal{D}_n\}}(B(x_r;\frac12r))\leq\mu_{\{p_n,\mathcal{D}_n\}}(B(x;r))\leq \mu_{\{p_n,\mathcal{D}_n\}}(B(x_r;\frac{3}{2}r)).
\end{equation}

\bigskip
\par Now we prove the necessity by a contradiction.
Without loss of generality, we assume $p_n=\#\mathcal{D}_n$ for any $n\geq 1$, which implies that $\sup_n p_n<\infty$.
Note that
$$
\#(B(x_r;2r)\cap\frac{\mathbb{Z}}{p_1\cdots p_{s_r+k_0}})\leq 4p_{s_r}\cdots p_{s_r+k_0}
$$
and
$$
\mu_{\{p_n,\mathcal{D}_n\}}=\delta_{p_1^{-1}\mathcal{D}_1}\ast\cdots\ast\delta_{(p_1\cdots p_{s_r+k_0})^{-1}\mathcal{D}_{s_r+k_0}}\ast\delta_{(p_1\cdots p_{s_r+k_0+1})^{-1}\mathcal{D}_{s_r+k_0+1}}\ast\cdots.
$$
Let $P:=\sup_ip_i$ and $T_1:=\{\sum_{j=s_r+k_0+1}^\infty\frac{d_j}{p_1\cdots p_{j}}:d_j\in \mathcal{D}_j\}$.
It follows from \eqref{(5.1)} and
$$
T\cap B(x_r;\frac{3}{2}r)\subset B(x_r,2r)\cap\frac{\mathbb{Z}}{p_1\cdots p_{s+k_0}}+T_1,
$$
that
$$
\mu_{\{p_n,\mathcal{D}_n\}}(B(x;r))\leq \mu_{\{p_n,D_n\}}(B(x_r;\frac{3}{2}r))\leq \frac{4p_{s_r}\cdots p_{s_r+k_0}}{p_1\cdots p_{s_r+k_0}}\leq4p_{s_r}r\leq 4Pr.
$$
Hence $\mathcal{H}^1(T)\geq \mu(T)/(4P)=\frac{1}{4P}$ by Theorem \ref{thm5.1},
which implies that $\mu_{\{p_n,\mathcal{D}_n\}}$ is not singular continuous with respect to the Lebesgue measure.

\bigskip
\par In the following, we prove the sufficiency.

Let $\{n:p_n>\#\mathcal{D}_n\}=\{n_1,n_2,\cdots,n_k,\cdots\}$ with $n_1<n_2<\cdots$.
It follows from \eqref{(5.1)} that
\begin{align*}
\mu_{\{p_n,\mathcal{D}_n\}}(B(x;r))\geq\mu_{\{p_n,\mathcal{D}_n\}}(B(x_r;\frac12r))&\geq \frac{1}{\#\mathcal{D}_1\#\mathcal{D}_2\cdots\#\mathcal{D}_{s_r+k_0}}
\\
&=\frac{1}{p_1\cdots p_{s_r+k_0}}\times
\frac{p_1\cdots p_{s_r+k_0}}{\#\mathcal{D}_1\#\mathcal{D}_2\cdots\#\mathcal{D}_{s_r+k_0}}
\\
&\geq \frac{r}{P^{k_0+1}}\prod_{j=1}^{k_r}\frac{p_{n_j}}{\#\mathcal{D}_{n_j}}
\\
&\geq \frac{r}{P^{k_0+1}}(1+\frac1N)^{k_r},
\end{align*}
where $n_1<n_2<\cdots<n_{k_r}\leq s_r+k_0<n_{k_r+1}$ and $N:=\sup_j\#\mathcal{D}_j$.
Let $c(r)=\frac{1}{P^{k_0+1}}(1+\frac1N)^{k_r}$.
It is easy to see that $c(r)\rightarrow\infty$ as $r\rightarrow0$.

Therefore, by Lemma \ref{thm5.2}, $\mathcal{H}^1(T)=0$.
The proof is complete.
\end{proof}

\begin{proof}[\bf Proof of Theorem \ref{thm2.2}]
The theorem is obtained directly from Theorems \ref{thm5.3}, \ref{thm2.1} and Lemma \ref{thm3.2}.
\end{proof}

\subsection{\emph{Fuglede's Conjecture on Cantor-Moran set}\label{subsect5.2}}
\

In this subsection, we consider Fuglede's Conjecture on Cantor-Moran set.
Specially, we provide the proof of Theorem \ref{thm2.3}.

It is well known that spectral and tiling properties do not change under similar transformations.
We provide a criterion that is effective for Moran sets(Lemma \ref{thm5.6}).
The following theorem is a basic tool for the spectrality of measure.
\begin{thm}[\cite{Jorgenson-Pederson_1998}]\label{thm6.1}
Let $\mu$ be a probability measure with a compact support, and let $\Lambda\subset \mathbb{R}^n$ be a countable subset.
Then $E_{\Lambda}$ is an orthogonal set of $L^{2}(\mu)$ if and only if
$$Q_\Lambda(\xi):=\sum_{\lambda\in\Lambda}|\widehat{\mu}(\xi+\lambda)|^2\leq1
$$
for $\xi\in\mathbb{R}^n$; and $Q_\Lambda(\xi)$ is an entire function in $\mathbb{C}^n$.
In particular, it is an orthogonal basis if and only if $Q_\Lambda(\xi)\equiv1$ for $\xi\in\mathbb{R}^n$.
\end{thm}

\begin{lem}\label{thm5.6}
Let $p_n$ be an integer with $|p_n|>1$, and let digit set $D_n,\widetilde{D}_n\subset\Bbb Z$ satisfy
$$\max\big\{\sum_{n=1}^\infty p_1^{-1}\cdots p_n^{-1}|d|:d\in D_n\big\}<\infty.
$$
Suppose that real number $Q\neq0$.
Then the associated Moran measure $\mu_{\{p_n,D_n\}}$ is a spectral measure with a spectrum $\Lambda$ if and only if $\mu_{\{p_n,Q^{-1}D_n\}}$ is a spectral measure with a spectrum $Q\Lambda$.
\end{lem}
\begin{proof}
Let $\Lambda$ be a spectrum of $\mu_{\{p_n,D_n\}}$.
It follows from Theorem \ref{thm6.1} that
$$
\sum_{\lambda\in\Lambda}|\widehat{\mu_{\{p_n,D_n\}}}(\xi+\lambda)|^2\equiv1
$$
for any $\xi\in\Bbb R$.
Then
\begin{align*}
\sum_{\lambda\in Q\Lambda}|\widehat{\mu_{\{p_n,Q^{-1}D_n\}}}(\xi+\lambda)|^2
&=\sum_{\lambda\in \Lambda}|\prod_{n=1}^\infty\widehat{\delta_{D_n}}(Q^{-1}p_n^{-1}\cdots p_1^{-1}(\xi+Q\lambda))|^2
\\
&=\sum_{\lambda\in \Lambda}|\prod_{n=1}^\infty\widehat{\delta_{D_n}}(p_n^{-1}\cdots p_1^{-1}(Q^{-1}\xi+\lambda))|^2
\equiv1
\end{align*}
for any $\xi\in\Bbb R$.
Using Theorem \ref{thm6.1} again, one has $Q\Lambda$ is a spectrum of $\mu_{\{p_n,Q^{-1}D_n\}}$.
\end{proof}
By directly applying the above lemma, we obtain the following result.
\begin{prop}\label{thm5.7}
Let $\{(p_n, \mathcal{D}_n, L_n)\}$ be a sequence of Hadamard triples with $\sup_n\{p_n^{-1}d:d\in\mathcal{D}_n\}<\infty$ and $\sup_n\#\mathcal{D}_n<\infty$, and let Cantor-Moran measure
$$
\mu_{\{p_n,\mathcal{D}_n\}}=\delta_{p_1^{-1}\mathcal{D}_1}\ast\delta_{p_1^{-1}p_2^{-1}\mathcal{D}_2}\ast\cdots,
$$
where $\gcd_{j\geq 1}\{\gcd \mathcal{D}_j\}=d$.
Then $\mu_{\{p_n,\mathcal{D}_n\}}$ is a spectral measure if and only if $\mu_{\{p_n,\mathcal{D}_n'\}}$ is a spectral measure,
where $\mathcal{D}_n'=\frac1d\mathcal{D}_n$ and $\gcd_{j\geq 1}\{\gcd \mathcal{D}_j'\}=1$.
\end{prop}
The above proposition ensures that our hypothesis $\gcd_{j\geq 1}\{\gcd \mathcal{D}_j\}=1$ of Theorem \ref{thm2.3} is reasonable.
In the following, we investigate Fuglede's Conjecture on the Cantor-Moran sets.
Recall that a fundamental domain of a lattice $\Bbb Z$ is a set $\Omega$ such that $\cup_{l\in \Bbb Z}(\Omega+l)$ tiles $\Bbb R$ almost everywhere, i.e., $\Omega\oplus\Bbb Z=\Bbb R$.

\begin{thm}\label{thm5.4}
Let $\{(p_n, \mathcal{D}_n, L_n)\}$ be a sequence of Hadamard triples with $\sup_n\{p_n^{-1}d:d\in\mathcal{D}_n\}<\infty$ and $\sup_n\#\mathcal{D}_n<\infty$, and let Cantor-Moran measure
$$
\mu_{\{p_n,\mathcal{D}_n\}}=\delta_{p_1^{-1}\mathcal{D}_1}\ast\delta_{p_1^{-1}p_2^{-1}\mathcal{D}_2}\ast\cdots.
$$
Suppose that $p_n=\#\mathcal{D}_n$ and $\gcd_{j\geq n}\{\gcd \mathcal{D}_j\}\equiv1$ for all $n\geq1$.
Then there exists a spectral set $T$ such that $\mu_{\{p_n,\mathcal{D}_n\}}$ is the restriction of the Lebesgue spectral measure on $T$.
In particular, $T$ tiles $\Bbb R$ with a tile set $\Bbb Z$, i.e., $T\oplus\Bbb Z=\Bbb R$.
\end{thm}
\begin{proof}
Theorems \ref{thm2.2} and \ref{thm5.3} tell us that $\mu_{\{p_n,\mathcal{D}_n\}}$ is an absolutely continuous spectral measure with a spectrum $\Lambda\subset\Bbb Z$.
From \cite{Dut_Lai}, we know that $\mu_{\{p_n,\mathcal{D}_n\}}$ is the restriction of the Lebesgue measure on some compact set $T=\text{spt}(\mu_{\{p_n,\mathcal{D}_n\}})$, i.e., $\mu_{\{p_n,\mathcal{D}_n\}}=C\chi_{T}dx$, where $C$ is a constant and $\chi_T$ is a characteristic function on $T$.
Due to $(p_n, \mathcal{D}_n, L_n)$ is a Hadamard triple, i.e., $\mathcal{D}_n=\{0,1,\cdots,p_n-1\}\pmod {p_n}$, we have $\mathscr{Z}(\widehat{\delta_{\mathcal{D}_n}})\supset\frac{\Bbb Z\setminus p_n\Bbb Z}{p_n}$.
This implies that
\begin{align}\label{(5.2)}
\mathscr{Z}(\widehat{\mu_{\{p_n,\mathcal{D}_n\}}})\supset\cup_{n=1}^\infty p_1\cdots p_n\frac{\Bbb Z\setminus p_n\Bbb Z}{p_n}=\Bbb Z\setminus\{0\}.
\end{align}
Hence $\Bbb Z$ is a bi-zero set of $\mu_{\{p_n,\mathcal{D}_n\}}$, which yields $\Lambda=\Bbb Z$.
It follows from \eqref{(5.2)} and Theorem 6.2 of \cite{JP92} (or Theorem 2.2 of \cite{LW97}) that $\mathcal{L}(T)=1$, where $\mathcal{L}(\cdot)$ is Lebesgue measure.
This yields that $C=1$, i.e., $\mu_{\{p_n,\mathcal{D}_n\}}$ is the restriction of the Lebesgue measure on $T$.
By Theorem 2.1 of \cite{LW97}, we obtain that $T$ is fundamental domain of $\Bbb Z$, i.e., $T\bigoplus\Bbb Z=\Bbb R$.
\end{proof}

\begin{proof}[\bf Proof of Theorem \ref{thm2.3}]
If there exists an equi-positive subsequence $\{\nu_{n_k}\}$,
from Theorems \ref{thm2.1} and \ref{thm5.4}, the support of $\mu_{\{p_n,\mathcal{D}_n\}}$ is a translational tile.
To show that $\mu_{\{p_n,\mathcal{D}_n\}}$ satisfies the no-overlap condition, we only need to prove
$$
\mu_{\{p_n,\mathcal{D}_n\}}((b+K_{>1})\cap(b'+K_{>1}))=0
$$
for any $b\neq b'\in K_1=p_1^{-1}\mathcal{D}_1$.
Applying Theorem \ref{thm5.4} to $\nu_2$,  one has
$\nu_2$ is an absolutely continuous with respect to the Lebesgue spectral measure and $\nu_2=\chi_{T}dx$, where $T$ is the support of $\nu_2$ and  $T\oplus\Bbb Z=\Bbb R$.
Note that
$$
\mu_{\{p_n,\mathcal{D}_n\}}=\delta_{p_1^{-1}\mathcal{D}_1}\ast\mu_{>1}=p_1\delta_{p_1^{-1}\mathcal{D}_1}\ast \chi_{p_1^{-1}T}dx=g(x)dx
$$
for some non-negative function $g(x)\in L^1(dx)$ whose support on $p_1^{-1}\mathcal{D}_1+p_1^{-1}T$.
According to Corollary 1.4 of \cite{Dut_Lai} and the spectrality of $\mu_{\{p_n,\mathcal{D}_n\}}$, there exists a constant $C>0$ such that $g(x)\equiv C$ on $x\in p_1^{-1}\mathcal{D}_1+p_1^{-1}T$.
Hence
$$
\int\chi_{(d+T)\cap(d'+T)}dx=0
$$
for any $d\neq d'\in \mathcal{D}_1$.
This implies that $$
\mu_{\{p_n,\mathcal{D}_n\}}((b+K_{>1})\cap(b'+K_{>1}))>0
$$
for any $b\neq b'\in p_1^{-1}\mathcal{D}_1$, i.e., $\mu_{\{p_n,\mathcal{D}_n\}}$ satisfies the no-overlap condition.

With loss of generality, we assume that the tiling of $\text{spt}(\nu_3)$ is unique(difference is a constant).
We prove the conclusion by a contradiction.
Suppose that $d>1$.
Write $\mathcal{D}_n'=\frac{1}{d}\mathcal{D}_n$ for all $n\geq2$.
Then $\gcd_{j\geq n}\{\gcd\mathcal{D}_j'\}\equiv1$ for any $n\geq2$.
It follows from Theorem \ref{thm5.4} that there exists $T_{2}\subset\Bbb R$ such that $T_{2}\oplus\Bbb Z=\Bbb R$ and
$$
\delta_{p_2^{-1}\mathcal{D}_2'}\ast\delta_{p_2^{-1}p_3^{-1}\mathcal{D}_3'}\ast\cdots=\chi_{T_{2}}dx,
$$
Hence
$$
\mu_{\{p_n,\mathcal{D}_n\}}=\delta_{p_1^{-1}\mathcal{D}_1}\ast C\chi_{p_1^{-1}dT_{2}}dx,
$$
for some $C>0$.
The no-overlap condition ensures
$$
\int\chi_{(d'+dT_{2})\cap(d''+dT_{2})}dx=0
$$
for any $d'\neq d''\in\mathcal{D}_1$.
This implies that $\frac{1}{d}\mathcal{D}_1+T_{2}$ can be written as a direct sum, that is,  $\frac{1}{d}\mathcal{D}_1+T_{2}=\frac{1}{d}\mathcal{D}_1\oplus T_{2}$.
If $\frac{1}{d}\mathcal{D}_1\oplus T_{2}$ is a tile, then there exists $0\in A\subset\Bbb R$ such that
\begin{align}\label{(5.3)}
A\oplus\frac{1}{d}\mathcal{D}_1\oplus T_{2}=\Bbb R.
\end{align}
Note that $
T_2=p_2^{-1}\mathcal{D}_2'\oplus p_2^{-1}T_3,
$
where $T_3$ is the support of $\delta_{p_3^{-1}\mathcal{D}_3'}\ast\delta_{p_3^{-1}p_4^{-1}\mathcal{D}_4'}\ast\cdots$.
Then \eqref{(5.3)} can be written as
\begin{align*}
T_3\oplus \mathcal{D}_2'\oplus p_2A\oplus\frac{p_2}{d}\mathcal{D}_1=\Bbb R.
\end{align*}
It follows from Theorem \ref{thm5.4} that $T_3\oplus \Bbb Z=\Bbb R$.
As $\text{spt}(\nu_3)=dT_3$ and the tiling of $\text{spt}(\nu_3)$ is unique, $\Bbb Z$ is a unique tiling of $T_3$.
Hence $\mathcal{D}_2'\oplus p_2A\oplus\frac{p_2}{d}\mathcal{D}_1=\Bbb Z+C$ for some $C\in\Bbb R$.
Since $0\in\mathcal{D}_2'\cap \mathcal{D}_1$, one has $C\in\Bbb Z$ and $p_2A\subset\Bbb Z$, which yields $\frac{p_2}{d}\mathcal{D}_1\subset\Bbb Z$.
According to $\gcd(\gcd\mathcal{D}_1,d)=1$, we know $d|p_2$, this is impossible.
In fact, by the proof of Lemma \ref{thm3.3}, $d|\mathcal{D}_2$ and $(p_2,\mathcal{D}_2,L_2)$ forms a Hadamard triple, the common divisor of $p_2$ and $d$ is $1$.
This completes the proof.
\end{proof}

\section{\bf Non-spectrality of Cantor-Moran measures\label{sect.6}}
\setcounter{equation}{0}
In this section, we study the non-spectrality of Cantor-Moran measures.
We need to emphasize that the condition
\begin{equation}\label{(6.1*)}
\sup_n\{p_n^{-1}d:d\in\mathcal{D}_n\}<\infty
\end{equation} plays an important role in this paper.

A natural question is:
$$\text{whether Cantor-Moran measure}\;\mu_{\{p_n,D_n\}}\;\text{is spectral with} \sup_n\{p_n^{-1}d:d\in\mathcal{D}_n\}=\infty?
$$
We show \eqref{(6.1*)} is indispensable but can be improved in the following examples.

\begin{exam}\label{exam6.1}
Let $p_n\equiv9$ and $\mathcal{D}_n=\{0,2,4^n\}$ for all $n\geq1$.
Then $\mu_{\{p_n,\mathcal{D}_n\}}$ is not a spectral measure.
\end{exam}
\begin{exam}\label{exam6.2}
Let $p_n\equiv3n$ and $\mathcal{D}_n=\{0,2,3n^2+1\}$ for all $n\geq1$.
Then $\mu_{\{p_n,\mathcal{D}_n\}}$ is a spectral measure.
\end{exam}
To prove the above examples, we need some known knowledge.
We give the concept of the selection map, which is used to characterize the structure of bi-zero sets of $\mu_{\{p_n,D_n\}}$ in Example \ref{exam6.1}.

Let $\Omega=\{0,1,2\}$ and let $\Omega^\ast=\cup_{k=0}^\infty\Omega^k$ be the set of finite words (by convention $\Omega^0=\{\emptyset\}$).
Denote by $I=i_1\cdots i_k$ an element in $\Omega^k$, and $|I|=k$ is its length.
For any $I,J\in\Omega^\ast,\;IJ$ is their natural conjunction.
In particular, $\emptyset I=I$, $I0^\infty=I00\cdots$ and $0^k=0\cdots0\in\Omega^k$.
\begin{defi}\label{defi6.1}
Let $3|q$. We call a map $\phi:\Omega^\ast\rightarrow\{-1,0,\cdots,q-2\}$ a selection mapping if
\\
$(i)\;\;\;\phi(\emptyset)=\phi(0^n)=0$ for all $n\geq1$;
\\
$(ii)$\;\;for any $I=i_1\cdots i_k\in\Omega^k,\phi(I)\in i_k+3\mathbb{Z}\cap \mathcal{C}$, where $\mathcal{C}=\{-1,0,\cdots,q-2\}$;
\\
$(iii)$\;for any $I=i_1\cdots i_k\in\Omega^\ast$,
there exists $J\in\Omega^\ast$ such that $\phi$ vanishes eventually on $IJ0^\infty$, $i.e.$, $\phi(IJ0^k)=0$ for sufficient large $k$.
\end{defi}

Let
$$
\Omega^\phi=\{I=i_1\cdots i_k\in\Omega^\ast:i_k\neq0,\;\phi(I0^n)=0\;\text{for\;sufficient\;large\;}n\}\cup\{\emptyset\},
$$
and let
$$
\phi^\ast(I)=\sum_{n=1}^\infty\phi(I0^\infty|_n)q^{n-1},\quad I\in\Omega^\phi.
$$
Here $I0^\infty|_n$ denotes the word of the first $n$ entries.

Dai, He and Lau \cite{DHL14} showed the following theorem.
\begin{thm}[\cite{DHL14}]\label{thm6.2}
Let $\rho=p/q\in(0,1)$ with $p,q\in\Bbb Z$, $3|q$ and $\gcd(p,q)=1$.
Suppose that $0\in\Lambda\subset\Bbb R$ is a countable set satisfying
$$(\Lambda-\Lambda)\setminus\{0\}\subset\cup_{j=1}^\infty\rho^{-j}(\pm\frac{1}{3}+\mathbb{Z}).$$
Then $\Lambda$ is maximal, i.e.,
$$(\Lambda\cup\{x\}-\Lambda\cup\{x\})\setminus\{0\}\not\subset\cup_{j=1}^\infty\rho^{-j}(\pm\frac{1}{3}+\mathbb{Z})
$$ for any $x\in\Bbb R\setminus\Lambda$,  if and only if there exist $m_0\geq1$ and a selection map $\phi$ such that $\Lambda=\rho^{-m_0}3^{-1}(\phi^\ast(\Omega^\phi))$.
\end{thm}
\begin{proof}[\bf Proof of Example \ref{exam6.1}]
Suppose that $\mu_{\{p_n,\mathcal{D}_n\}}$ is a spectral measure with a spectrum $\Lambda$.
The property of the bi-zero set gives that
$$
2(\Lambda-\Lambda)\setminus\{0\}\subset \mathscr{Z}(\widehat{\mu_{\{p_n,\mathcal{D}_n\}}})=\cup_{j=1}^\infty 9^j(\pm\frac{1}{3}+\mathbb{Z}).
$$
Theorem \ref{thm6.2} gives a selection mapping $\phi$ and a subset $\Omega'\subset\Omega^\phi$ such that
$$
\Lambda=6^{-1}9^{m_0}\phi^*(\Omega')
$$
for some $m_0\geq1$.
Write $\mu_{\{p_n,\mathcal{D}_n\}}=\mu_n\ast\mu_{>n}$, where
$$
\mu_n:=\delta_{p_1^{-1}\mathcal{D}_1}\ast\delta_{(p_2p_1)^{-1}\mathcal{D}_2}\ast\cdots \ast\delta_{(p_1\cdots p_n)^{-1}\mathcal{D}_n}.
$$
The following claims are needed.
\\
\textbf{Claim\;1:}There exist $C,\beta>0$ such that
$$
\prod_{j=1}^\infty\frac{1}{3}\big|1+e^{2\pi i\frac{2y}{9^{j}}}+e^{2\pi i(\frac{4}{9})^jx}\big|\leq C(\ln |x|)^{-\beta},\;|x|>1,\;y\in\Bbb R.
$$
Let integer $b$ satisfy $b\geq2\beta^{-1}$, and $\mathbf{I}_n=\{\mathbf{i}\in\Omega':|\mathbf{i}|\leq n^b\}$.
\\
\textbf{Claim\;2:} $$
\sum_{\mathbf{i}\in \mathbf{I}_n}|\widehat{\mu_{n^b+m_0-1}}(\xi+6^{-1}9^{m_0}\phi^*(\mathbf{i}))|^2\leq 1
$$
for any $n\geq1$.

In the following, we use the above claims to prove our conclusion.
Let $Q_n(\xi):=\sum_{\mathbf{i}\in \mathbf{I}_n}|\widehat{\mu}(\xi+6^{-1}9^{m_0}\phi^*(\mathbf{i}))|^2.$
For any $\mathbf{i}\in \mathbf{I}_{n+1}\setminus \mathbf{I}_n$ and $|\xi|<\frac{1}{2}$, by \textbf{Claim\;1} and
$$
|\xi+6^{-1}9^{m_0}\phi^*(\mathbf{i})|\geq 6^{-1}9^{m_0}(9^{n^b}-7\times 9^{n^b-1}-\cdots)\geq 6^{-1}9^{n^b+m_0-1},\;\;\xi\in(0,1)
$$
there exist $n_0$ and some constants $C,C'>0$ such that
\begin{align*}
|\widehat{\mu_{>(n+1)^b+m_0-1}}(\xi+\lambda(\mathbf{i}))|&=\prod_{j=1}^\infty\frac{1}{3}\left|1+e^{2\pi i\frac{2(\xi+6^{-1}9^{m_0}\phi^*(\mathbf{i}))}{9^{(n+1)^b+j+m_0-1}}}+e^{2\pi i(\frac{4}{9})^{(n+1)^b+j+m_0-1}(\xi+6^{-1}9^{m_0}\phi^*(\mathbf{i}))}\right|
\\
&\leq C(\ln|6^{-1}4^{(n+1)^b}9^{n^b-(n+1)^b}|)^{-\beta}
\\
&\leq C' n^{-\frac{1}{2}b \beta}\leq C' n^{-1}
,
\end{align*}
for any $\mathbf{i}\in \mathbf{I}_{n+1}\setminus \mathbf{I}_n$ and $n\geq n_0$.
According to the above estimate and \textbf{Claim\;2}, one has
\begin{align*}
Q_{n+1}(\xi)&=Q_n(\xi)+\sum_{I\in \mathbf{I}_{n+1}\setminus \mathbf{I}_n}|\widehat{\mu_{\{p_n,\mathcal{D}_n\}}}(\xi+6^{-1}9^{m_0}\phi^*(\mathbf{i}))|^2
\\
&\leq Q_n(\xi)+\frac{C'^2}{n^2}\sum_{\mathbf{i}\in \mathbf{I}_{n+1}\setminus\mathbf{I}_n}|\widehat{\mu_{(n+1)^b+m_0-1}}(\xi+6^{-1}9^{m_0}\phi^*(\mathbf{i}))|^2
\\
&\leq Q_n(\xi)+\frac{C'^2}{n^2}(1-Q_n(\xi)).
\end{align*}
Therefore,
$$
1-Q_{n+1}(\xi)\geq(1-Q_n(\xi))(1-\frac{C'^2}{n^2})\geq(1-Q_{n_0}(\xi))\prod_{j=n_0}^n(1-\frac{C'^2}{j^2})
$$
for all $n\geq n_0$.
Letting $n\rightarrow\infty$, then
$$
1-Q(\xi)\geq(1-Q_{n_0}(\xi))\prod_{j=n_0}^\infty(1-\frac{C'^2}{j^2}).
$$
From $\prod_{j=n_0}^\infty(1-\frac{C'^2}{j^2})\neq0$, we obtain that $Q(\xi)\not\equiv1$.
By Theorem \ref{thm6.1}, $\Lambda$ is not a spectrum of $\mu_{\{p_n,\mathcal{D}_n\}}$, a contradiction.

Finally, we complete our proof by showing \textbf{Claims 1,2}.
We first prove \textbf{Claim 1}.
Define that $|\langle x\rangle|:=|x_1|$, where $x=x_1+x_2$ with $x_2\in\mathbb{Z}$ and $x_1\in[-\frac{1}{2},\frac{1}{2})$.
Let
$$
b_0:=\max_{|\langle x\rangle|\geq\frac{1}{18}}\{\frac{1}{3}\big|1+e^{2\pi ix}+e^{2\pi iy}\big|\}\leq\frac{1}{3}\big|1+|1+e^{2\pi i\frac{1}{18}}|\big|<1.
$$
Then
\begin{equation}\label{(6.1)}
\prod_{j=1}^\infty\frac{1}{3}\left|1+e^{2\pi i\frac{2y}{9^{j}}}+e^{2\pi i(\frac{4}{9})^jx}\right|\leq b_0^{l},
\end{equation}
where $l=\#\{j:|\langle (\frac{4}{9})^jx\rangle|\geq\frac{1}{18}\}$.
Put
$$
\left\{j:\left|\left\langle\left(\frac{4}{9}\right)^jx\right\rangle\right|\geq\frac{1}{18}\right\}=\{j_1,j_2,\cdots, j_l\}
$$
with $j_1<j_2<\cdots<j_l$ and $|(\frac{4}{9})^{j_l+1}x|\leq1$.
For any $|(\frac{4}{9})^jx|>1$, we can write
$$
\left(\frac{4}{9}\right)^jx=9^{s_j}(\mathbb{Z}\setminus9\mathbb{Z})+\epsilon\in 9^{s_j}(\mathbb{Z}\setminus9\mathbb{Z})+\left[-\frac{1}{2},\frac{1}{2}\right),
$$
which implies that
\begin{equation}\label{(6.2)}
\left|\left(\frac{4}{9}\right)^{j-1}x\right|\geq 9^{s_j}.
\end{equation}
If $|\langle(\frac{4}{9})^jx\rangle|<\frac{1}{18}$,
then $|\langle(\frac{4}{9})^{j+s_j+1}x\rangle|\geq\frac{1}{9}-\frac{1}{18}=\frac{1}{18}$.
Without loss of generality, assuming $j_{i+1}\geq j_i+2$, then we have $|\langle(\frac{4}{9})^{j_i+1}x\rangle|<\frac{1}{18}$.
From \eqref{(6.2)}, a simple calculation gives
$$
\left|\left(\frac{4}{9}\right)^{j_{i+1}}x\right|\geq \left|\left(\frac{4}{9}\right)^{j_{i}+s_{j_i+1}+2}x\right|\geq\left(\frac{4}{9}\right)^2\left|\left(\frac{4}{9}\right)^{j_i}x\right|^{\frac{\ln4}{\ln9}}.
$$
Hence
\begin{equation}\label{(6.3)}
\frac{9}{4}\geq\left|\left(\frac{4}{9}\right)^{j_{l}}x\right|\geq \left(\frac{4}{9}\right)^{2}\left|\left(\frac{4}{9}\right)^{j_{l-1}}x\right|^{\frac{\ln4}{\ln9}}\geq\cdots\geq9^{-2}|x|^{(\frac{\ln4}{\ln9})^l}.
\end{equation}
Let $\alpha:=\frac{\ln4}{\ln9}$ and $\beta:=\frac{\ln b_0}{\ln \alpha}$.
By combining \eqref{(6.1)} with \eqref{(6.3)}, we deduce
$$
\prod_{j=1}^\infty\frac{1}{3}\big|1+e^{2\pi i\frac{2y}{9^{j}}}+e^{2\pi i(\frac{4}{q})^jx}\big|\leq b_0^{l}\leq \alpha^{l\beta}\leq \left(\ln\frac{9^3}{4}\right)^{\beta}(\ln|x|)^{-\beta},
$$
which completes the proof of \textbf{Claim\;1}.

According to the definition of $\phi^*$, we obtain that
$$
6^{-1}9^{m_0}\phi^*(\mathbf{i})-6^{-1}9^{m_0}\phi^*(\mathbf{j})\in\mathscr{Z}(\widehat{\delta_{(p_1\cdots p_{t})^{-1}D_t}})\subset\mathscr{Z}(\widehat{\mu_{n+m_0-1}}),
$$
where $t=m_0-1+\min\{j:\mathbf{i}|_j\neq\mathbf{j}|_j\}$.
Then $\{6^{-1}9^{m_0}\phi^*(\mathbf{i})\}_{\mathbf{i}\in \mathbf{I}_n}$ is a bi-zero set of $\mu_{n+m_0-1}$.
Then Theorem \ref{thm6.1} gives \textbf{Claim\;2}, which completes the proof.
\end{proof}

\begin{proof}[\bf Proof of Example \ref{exam6.2}]
Write
$$
\mu_{\{p_n,\mathcal{D}_n\}}:=\mu_n\ast\mu_{>n},
$$
where $\mu_n=\delta_{p_1^{-1}\mathcal{D}_1}\ast\cdots\delta_{(p_1\cdots p_n)^{-1}\mathcal{D}_n}$.
Let $\Lambda_n=\sum_{j=1}^n 3^jj!\{0,\pm\frac{1}{3}\}$ and $\Lambda=\cup_{n=1}^\infty\Lambda_n$.
For any $\lambda\neq\lambda'\in\Lambda_n$,
we know that
$$
\lambda-\lambda'=\sum_{j=1}^n3^jj!(w_j-w_j')=\sum_{j=j_0}^n3^jj!(w_j-w_j')\in 3^{j_0}j_0!(\pm\frac13+\Bbb Z)
\subset\mathscr{Z}(\widehat{\delta_{p_1^{-1}\cdots p_{j_0}^{-1}\mathcal{D}_{j_0}}})\subset\mathscr{Z}(\widehat{\mu_n}),
$$
where $w_j,w_j'\in\{\pm\frac13\}$ and $j_0:=\min\{j:w_j\neq w_j'\}\leq n$.
Hence $\Lambda_n$ is a bi-zero set of $\mu_n$. Note that $\#\Lambda_n=\#\text{spt}\mu_n$.
Then $\{e^{2\pi i\lambda x}:\lambda\in\Lambda_n$ is complete in $L^2(\mu_n)$, i.e., $\Lambda_n$ is a spectrum of $\mu_n$.

For any $\lambda=\sum_{j=1}^n 3^jj!w_j$ with $w_j\in\{0,\pm\frac{1}{3}\}$, we have
$$
|x+\lambda|\leq \frac{1}{3}(3^nn!+3^{n-1}(n-1)!+\cdots+1)\leq\frac{3^nn!}{2}
$$
for any $|x|<\frac{1}{3}$.
Note that
$$
\prod_{j=1}^3\frac{1}{3}\left|1+e^{2\pi i\frac{2(x+\lambda)}{3^{n+j}(n+j)!}}+e^{2\pi i(\frac{(3(n+j)^2+1)(x+\lambda)}{3^{n+j}(n+j)!})}\right|\geq\varepsilon_0^3,
$$
where
$\varepsilon_0:=\min_{|x|<\frac{1}{6}}\frac{1}{3}|1+e^{2\pi ix}+e^{2\pi iy}|.$
Similar to Lemma \ref{thm3.9}, we have
$$
|\widehat{\mu_{>n}}(x+\lambda)|\geq\prod_{j=1}^\infty\frac{1}{3}\left|1+e^{2\pi i\frac{2(x+\lambda)}{3^{n+j}(n+j)!}}+e^{2\pi i(\frac{(3(n+j)^2+1)(x+\lambda)}{3^{n+j}(n+j)!})}\right|\geq \varepsilon_0^3\prod_{j=1}^\infty\cos\frac{1}{3^{j+3}}>0.
$$
Hence the example follows from Theorem 2.3(i) in \cite{AHH19}.
\end{proof}

\bigskip

\bigskip

\noindent {\bf Funding}.
The first author was supported by NSFC (Grant No. 11925107, 12226334), National Key R\&D Program of China (Grant No. 2021YFA1003100) and Key Research Program of Frontier Sciences, CAS (Grant No. ZDBS-LY-7002), and the second author was supported in part by the NNSF of China (Nos. 11831007 and 12071125), China Postdoctoral Science Foundation (No. 2022TQ0358 and 2022M723330).

\end{document}